\documentclass[]{interact}

\usepackage{epstopdf}
\usepackage[caption=false]{subfig}

\usepackage{url}

\usepackage[numbers,sort&compress]{natbib}
\bibpunct[, ]{[}{]}{,}{n}{,}{,}
\renewcommand\bibfont{\fontsize{10}{12}\selectfont}

\theoremstyle{plain}
\newtheorem{theorem}{Theorem}[section]
\newtheorem{lemma}[theorem]{Lemma}

\theoremstyle{definition}
\newtheorem{definition}[theorem]{Definition}
\newtheorem{assumption}[theorem]{Assumption}

\theoremstyle{remark}
\newtheorem{remark}{Remark}

\usepackage{algpseudocode}
\usepackage{algorithm}

\newcommand{\R}{\mathbb{R}}

\newcommand{\N}{\mathbb{N}}

\newcommand{\bigO}{\mathcal{O}}
\DeclareMathOperator*{\argmin}{arg\,min}

\newcommand{\grad}{\nabla}
\renewcommand{\Pr}{\mathbb{P}}

\usepackage{bm}
\renewcommand{\b}[1]{\bm{#1}} 

\newcommand{\bd}{\b{d}}

\newcommand{\bg}{\b{g}}

\newcommand{\bs}{\b{s}}
\newcommand{\bv}{\b{v}}
\newcommand{\bx}{\b{x}}
\newcommand{\by}{\b{y}}

\usepackage[svgnames]{xcolor}
\newcommand{\LR}[1]{{#1}}
\newcommand{\LRnew}[1]{{#1}}

\newcommand{\kappaef}{\kappa_{\textnormal{ef}}}
\newcommand{\kappaeg}{\kappa_{\textnormal{eg}}}
\newcommand{\kappaeh}{\kappa_{\textnormal{eh}}}
\newcommand{\overgammainc}{\overline{\gamma}_{\textnormal{inc}}}
\newcommand{\gammainc}{\gamma_{\textnormal{inc}}}
\newcommand{\gammadec}{\gamma_{\textnormal{dec}}}
\newcommand{\flow}{f_{\textnormal{low}}}
\newcommand{\padd}{p_{\textnormal{add}}}
\newcommand{\pdrop}{p_{\textnormal{drop}}}

\begin{document}

\articletype{}

\title{Randomized Subspace Derivative-Free Optimization with Quadratic Models and Second-Order Convergence}

\author{
\name{Coralia Cartis\textsuperscript{a} and Lindon Roberts\thanks{CONTACT L.~Roberts. Email: lindon.roberts@sydney.edu.au}\textsuperscript{b}}
\affil{\textsuperscript{a}Mathematical Institute, University of Oxford, Radcliffe Observatory Quarter, Woodstock Road, Oxford OX2 6GG, UK; \textsuperscript{b}School of Mathematics and Statistics, University of Sydney, Carslaw Building, Camperdown NSW 2006, Australia}
}

\maketitle

\begin{abstract}
We consider model-based derivative-free optimization (DFO) for large-scale problems, based on iterative minimization in random subspaces.
We provide the first worst-case complexity bound for such methods for convergence to approximate second-order critical points, and show that these bounds have significantly improved dimension dependence compared to standard full-space methods, provided low accuracy solutions are desired and/or the problem has low effective rank.
We also introduce a practical subspace model-based method suitable for general objective minimization, based on iterative quadratic interpolation in subspaces, and show that it can solve significantly larger problems than state-of-the-art full-space methods, while also having comparable performance on medium-scale problems when allowed to use full-dimension subspaces.
\end{abstract}


\section{Introduction}

In this work, we consider the problem of smooth nonlinear optimization where derivatives of the objective are unavailable, so-called derivative-free optimization (DFO) \cite{Conn2009,Audet2017,Larson2019}.
DFO methods are widely used across many fields \cite{Alarie2021}, but in this work we focus on large-scale derivative-free problems, such as those arising from bilevel learning of variational models in imaging \cite{Ehrhardt2021}, fine-tuning large language models \cite{Malladi2023}, adversarial example generation for neural networks \cite{Ughi2022}, and as a possible proxy for global optimization methods in the high-dimensional setting \cite{Cartis2022Escaping}. 
Here we focus on the widely used class of model-based DFO methods, which are based on iteratively constructing local (interpolating) models for the objective.

Standard model-based methods \cite{Conn2009,Larson2019} suffer from a lack of scalability with respect to the problem dimension, coming from dimension-dependent interpolation error bounds affecting worst-case complexity, and significant per-iteration linear algebra costs;  thus, they have traditionally been limited to small- or medium-scale problems.
More recently, starting with \cite{Cartis2023}, model-based DFO methods based on iterative minimization in random subspaces has proven to be a promising approach for tackling large-scale DFO problems in theory and practice.
Theoretically, for example, using linear models in random subspaces improves the worst-case complexity of such methods from $O(n^3 \epsilon^{-2})$ to $O(n^2 \epsilon^{-2})$ evaluations to find an $\epsilon$-approximate stationary point, for a problem with $n$ decision variables.
This aligns with similar results for derivative-based optimization \cite{Cartis2022Zhen,Shao2021} and other DFO approaches, such as direct search methods with random models \cite{Gratton2015,Roberts2023} and for gradient sampling in randomly drawn subspaces \cite{Kozak2021}.
This model-based framework has been extended to handle stochastic objective values \cite{Dzahini2024b} and specialized to exploit structured \LR{quadratic} model construction approaches \cite{Chen2024}.
The case of direct search methods in random subspaces has also been extended to the stochastic and nonsmooth case, and to show convergence to second-order optimal points \cite{Dzahini2024}.

\LRnew{There are alternative approaches for large-scale DFO where iterations are performed in subspaces that are deterministically generated from information from previous iterations, such as \cite{Zhang2012,Kimiaei2023,VanDeBerg2024}, where the method from \cite[Chapter 5.3]{Zhang2012} alternates between iterations in two-dimensional subspaces and full-dimensional finite differencing. The method from \cite{Xie2024} uses partially random subspaces, similar to \cite{Chen2024}, but is specialized to two-dimensional subspaces. We also note \cite{Menickelly2023}, where a full-dimensional model is updated at each iteration only inside a randomly generated subspace.}

\subsection{Contributions}
Existing random subspace methods for model-based DFO have theoretical convergence to approximate first-order optimal points, but there is currently no theory handling second-order optimality in random subspace model-based DFO (with only \cite{Dzahini2024} handling the direct search case).
Although \cite{Chen2024} does provide a mechanism for constructing fully quadratic interpolation models in random subspaces, which is a necessary component of second-order convergence theory, only first-order convergence theory is given.
We also note the related work \cite{Blanchet2019}, which proves convergence to second-order optimal points for a trust-region method with probabilistic models and function values, which uses similar ideas but with different sources of uncertainty (instead of the approach here which uses randomly drawn subspaces with deterministic models and function values).

There are two main contributions in this work:
\begin{itemize}
    \item First, we show how the random subspace model-based method in \cite{Cartis2023} can be extended to allow convergence to second-order stationary points (similar to full space methods \cite{Conn2009,Garmanjani2015}), and prove a high probability worst-case complexity bound for this method. Our approach is based on adapting a recent approach for second-order convergence of (derivative-based) subspace adaptive regularization with cubics \cite{Shao2021,Cartis2024}.
    We show a worst-case complexity bound\footnote{The notation $\tilde{O}(\cdot)$ is the same as $O(\cdot)$ but hiding logarithmic factors as well as constants.} of $\tilde{O}(n^{4.5} \epsilon^{-3})$ iterations and evaluations to achieve second-order optimality level $\epsilon$ for an $n$-dimensional problem, a significant improvement on the $O(n^9 \epsilon^{-3})$ iterations and $O(n^{11} \epsilon^{-3})$ evaluations for full space methods.
    For this theory to hold, we need to consider either low accuracy solutions or problems with low effective rank (i.e.~low rank Hessians, such as fine-tuning large language models \cite{Malladi2023}).
    We note that the numerical evidence from \cite{Cartis2023} suggests that subspace methods are indeed most practical for finding low accuracy solutions.
    \item We present a new practical algorithm for random subspace model-based DFO based on quadratic interpolation models.
    In \cite{Cartis2023}, a practical approach based on linear interpolation models was built for nonlinear least-squares problems (a special case where linear models provide good practical performance \cite{Cartis2019}).
    Here, we extend this approach to handle general quadratic interpolation approaches, focusing in particular on underdetermined quadratic interpolation \cite{Powell2004} in subspaces.
    Compared to state-of-the-art full-space methods using similar model constructions \cite{Cartis2019b}, our approach has comparable performance when set to use full-dimensional subspaces, but when using low-dimensional subspaces is able to solve larger problems (of dimension $n\approx O(1000)$) where existing approaches fail due to the significant linear algebra costs at this scale.
\end{itemize}

\paragraph*{Structure}
Firstly, in Section~\ref{sec_background} we recall the  random subspace model-based approach from \cite{Cartis2023} and its associated first-order complexity results.
In Section~\ref{sec_second_order} we extend this approach to find second-order critical points, stating the theoretical algorithm and proving the associated worst-case complexity bounds.
Then, in Section~\ref{sec_implementation} we outline our practical subspace method based on quadratic model construction in subspaces, which we test numerically in Section~\ref{sec_numerics}.

\paragraph*{Notation}
We use $\|\cdot\|$ to denote the Euclidean norm of vectors in and operator 2-norm of matrices, and $B(\bx,\Delta) = \{\by : \|\by-\bx\| \leq \Delta\}$ is the closed ball of radius $\Delta$ centered at $\bx$.
We use $\tilde{O}(\cdot)$ to denote $O(\cdot)$ with both logarithmic and constant factors omitted.

\section{Background} \label{sec_background}
In this work we consider the problem
\begin{align}
    \min_{\bx\in\R^n} f(\bx), \label{eq_problem}
\end{align}
where $f:\R^n\to\R$ is continuously differentiable and nonconvex, but $\grad f$ is not available.
In standard model-based DFO, at iteration $k$, we construct a local quadratic approximation for $f$ that we hope to be accurate near the current iterate $\bx_k$:
\begin{align}
    f(\bx_k+\bs) \approx m_k(\bs) := f(\bx_k) + \bg_k^T \bs + \frac{1}{2}\bs^T H_k \bs,
\end{align}
for some $\bg_k\in\R^n$ and (symmetric) $H_k\in\R^{n\times n}$.
This is typically achieved by finding $\bg_k$ and $H_k$ to interpolate the value of $f$ at a collection of points near  $\bx_k$.
We then use this model inside a trust-region framework to ensure a globally convergent algorithm, where---compared to derivative-based trust-region methods---special care has to be taken to ensure that $m_k$ is sufficiently accurate and $\|\bg_k\|$ is not too small relative to $\Delta_k$, see e.g.~\cite{Conn2009} for details.

In the large-scale case, which would typically be $n\geq 1000$ for DFO, the linear algebra cost of model construction and management becomes significant, and the model error (relative to the true objective value) can grow rapidly.
To avoid this issue, in \cite{Cartis2023} it was proposed to perform each iteration within a randomly drawn, low-dimensional subspace. 
Specifically, for some $p\ll n$, at each iteration $k$ we generate a random $P_k\in\R^{n\times p}$, defining an affine space $\mathcal{Y}_k := \{\bx_k+P_k \hat{\bs} : \hat{\bs}\in\R^p\}$.
We then construct a low-dimensional model to approximate $f$ only on $\mathcal{Y}_k$,
\begin{align}
    f(\bx_k+P_k \hat{\bs}) \approx \hat{m}_k(\hat{\bs}) := f(\bx_k) + \hat{\bg}_k^T \bs + \frac{1}{2} \hat{\bs}^T \hat{H}_k \hat{\bs}, \label{eq_model_subspace}
\end{align}
for some $\hat{\bg}_k\in\R^p$ and (symmetric) $\hat{H}_k\in\R^{p\times p}$.
We will adopt the convention where variables with hats are in the low-dimensional space and those without hats are in the full space.
The trust-region framework is easily adapted to this situation: the model $\hat{m}_k$ is minimized subject to a ball constraint to get a tentative low-dimensional step $\hat{\bs}_k$, which in turn defines a new potential iterate $\bx_k + P_k\hat{\bs}_k$, which may be accepted or rejected via a sufficient decrease test.

To achieve first-order convergence, the model $\hat{m}_k$ \eqref{eq_model_subspace} must be sufficiently accurate in the following sense, \cite[Definition 1]{Cartis2023}.

\begin{definition} \label{def_subspace_fully_linear}
    Given $\bx\in\R^n$, $\Delta>0$ and $P\in\R^{n\times p}$, a model $\hat{m}:\R^p\to\R$ is $P$-fully linear in $B(\bx,\Delta)$ if
    \begin{subequations}
        \begin{align}
            |f(\bx+P\hat{\bs}) - \hat{m}(\hat{\bs})| &\leq \kappaef \Delta^2, \\
            \|P^T \grad f(\bx+P\hat{\bs}) - \grad \hat{m}(\hat{\bs})\| &\leq \kappaeg \Delta,
        \end{align}
    \end{subequations}
    for all $\hat{\bs}\in\R^p$ with $\|\hat{\bs}\|\leq\Delta$, where the constants $\kappaef,\kappaeg>0$ are independent of $\bx$, $\Delta$, $P$ and $\hat{m}$.
\end{definition}

\begin{algorithm}[tbh]
\begin{algorithmic}[1]
\Statex \textbf{Inputs:} initial iterate $\bx_0 \in \R^n$, trust-region radii $0 < \Delta_0 \leq \Delta_{\max}$, trust-region updating parameters $0 < \gammadec < 1 < \gammainc$, subspace dimension $p\in\{1,\ldots,n\}$, and success check parameters $\eta\in(0,1)$ and $\mu>0$.
\vspace{0.2em}
\For{$k=0,1,2,\ldots$}
   \State  \LR{Generate a random matrix $P_k\in\R^{n\times p}$ satisfying Assumption \ref{ass_well_aligned}.}
    \State \LR{Build a model $\hat{m}_k$ \eqref{eq_model_subspace} which is $P_k$-fully linear in $B(\bx_k,\Delta_k)$.}
    \State Calculate a step by approximately solving the trust-region subproblem
    \begin{align}
        \hat{\bs}_k \approx \argmin_{\hat{\bs}\in\R^p} \hat{m}_k(\hat{\bs}), \quad \text{s.t.} \:\: \|\hat{\bs}\|\leq\Delta_k. \label{eq_trs}
    \end{align}
    \State Evaluate $f(\bx_k+P_k\hat{\bs}_k)$ and calculate the ratio
    \begin{align}
        R_k := \frac{f(\bx_k) - f(\bx_k+P_k\hat{\bs}_k)}{\hat{m}_k(\bm{0}) - \hat{m}_k(\hat{\bs}_k)}. \label{eq_ratio_test}
    \end{align}
    \If{$R_k\geq\eta$ and $\|\hat{\bg}_k\| \geq \mu \Delta_k$}
        \State (\textit{successful}) Set $\bx_{k+1}=\bx_k+P_k\hat{\bs}_k$ and $\Delta_{k+1}=\min(\gammainc\Delta_k,\Delta_{\max})$.
    \Else 
        \State (\textit{unsuccessful}) Set $\bx_{k+1}=\bx_k$ and $\Delta_{k+1}=\gammadec\Delta_k$.
    \EndIf
\EndFor
\end{algorithmic}
\caption{Random subspace DFO algorithm applied to \eqref{eq_problem}, a simplified version of \cite[Algorithm 1]{Cartis2023} based on \cite[Algorithm 2.1]{Dzahini2024b}.}
\label{alg_rsdfo}
\end{algorithm}

A prototypical random subspace algorithm is presented in Algorithm~\ref{alg_rsdfo}.
It is a simplification of the original algorithm \cite[Algorithm 1]{Cartis2023}, which includes a more sophisticated trust-region management procedure based on \cite{Powell2009} and allows the model $\hat{m}_k$ to not be $P_k$-fully linear at all iterations.
The version presented here is based on \cite[Algorithm 2.1]{Dzahini2024b} and has the same first-order convergence guarantees.

In order to prove convergence of Algorithm~\ref{alg_rsdfo}, we need the following standard assumptions on the objective function, model and trust-region subproblem calculation \eqref{eq_trs}.

\begin{assumption} \label{ass_smoothness}
    The objective function $f:\R^n\to\R$ in \eqref{eq_problem} is bounded below by $\flow$ and is continuously differentiable, where $\grad f$ is $L_{\grad f}$-Lipschitz continuous for some $L_{\grad f}>0$.
\end{assumption}

\begin{assumption} \label{ass_bounded_hess}
    The model Hessians $\hat{H}_k$ are uniformly bounded, $\|\hat{H}_k\|\leq \kappa_H$ for all $k$, for some $\kappa_H\geq 1$.
\end{assumption}

\begin{assumption} \label{ass_cauchy_decrease}
    The computed solution $\hat{\bs}_k$ in \eqref{eq_trs} satisfies the Cauchy decrease condition
    \begin{align}
        \hat{m}_k(\bm{0}) - \hat{m}_k(\hat{\bs}_k) \geq c_1 \|\hat{\bg}_k\| \min\left(\Delta_k, \frac{\|\hat{\bg}_k\|}{\max(\|\hat{H}_k\|, 1)}\right),
    \end{align}
    for some $c_1\in[1/2,1]$ independent of $k$.
\end{assumption}

While the above assumptions are common for trust-region methods, we need an additional assumption on the random matrices $P_k$ defining the subspace at each iteration.

\begin{assumption} \label{ass_well_aligned}
    The matrices $P_k$ are uniformly bounded, $\|P_k\|\leq P_{\max}$ for all $k$, and at each iteration $P_k$ is \emph{well-aligned}, i.e.~
    \begin{align}
        \Pr\left[\|P_k^T \grad f(\bx_k)\| \geq \alpha \|\grad f(\bx_k)\|\right] \geq 1-\delta, \label{eq_well_aligned}
    \end{align}
    for some $\alpha>0$ and $\delta\in(0,1)$ independent of $k$ and all previous iterations $P_0,\ldots,P_{k-1}$.
\end{assumption}

Although  condition \eqref{eq_well_aligned} depends on $\grad f(\bx_k)$, which is not available, Assumption~\ref{ass_well_aligned} can still be satisfied using \emph{Johnson-Lindenstrauss transforms} \cite{Woodruff2014}.
For example, $P_k$ may have independent $N(0,1/p)$ entries or be the first $p$ columns of a random orthogonal $n\times n$ matrix scaled by $\sqrt{n/p}$.
More details on how Assumption~\ref{ass_well_aligned} may be satisfied can be found in \LR{\cite[Chapter 4]{Shao2021}, \cite[Section 3.2]{Cartis2022Zhen},} \cite[Section 2.6]{Cartis2023}, \cite[Section 3.1]{Dzahini2024b}, \cite[Section 3.3]{Roberts2023}, or \cite[Section 2.5]{Chen2024}, for example.
Most crucially, for such constructions, \eqref{eq_well_aligned} can be satisfied for $p=\bigO(1)$ independent of $n$.

Under these assumptions, Algorithm~\ref{alg_rsdfo} has a first-order worst-case complexity bound of the following form (e.g.~\cite[Corollary 1]{Cartis2023} or \cite[Theorem 9]{Chen2024}).

\begin{theorem} \label{thm_first_order_complexity}
    Suppose Assumptions~\ref{ass_smoothness}, \ref{ass_bounded_hess}, \ref{ass_cauchy_decrease} and \ref{ass_well_aligned} hold.
    Suppose also that $\gammainc > \gammadec^{-2}$ and $\delta$ sufficiently small in Assumption~\ref{ass_well_aligned}.
    Then for $k$ sufficiently large, the iterates of Algorithm~\ref{alg_rsdfo} satisfy
    \begin{align}
        \Pr\left[\min_{j\leq k} \|\grad f(\bx_j)\| \leq \frac{C\sqrt{\kappa_{H}}\, \kappa_d}{\sqrt{k}}\right] \geq 1-e^{-ck},
    \end{align}
    for constants $c,C>0$, where $\kappa_d :=\max(\kappaef,\kappaeg)$ from Definition~\ref{def_subspace_fully_linear}.
\end{theorem}

This result implies that $\|\grad f(\bx_k)\|$ is driven below $\epsilon$ for the first time after $\bigO(\kappa_H \kappa_d^2 \epsilon^{-2})$ iterations and $\bigO(p \kappa_H \kappa_d^2 \epsilon^{-2})$ objective evaluations (with high probability), assuming $\bigO(p)$ interpolation points are used to construct $\hat{m}_k$.
Assuming linear interpolation models, we have $\kappa_d=\bigO(P_{\max}^2 p)$ and $\kappa_H=1$, so in the full-space case $p=n$ and $P_k=I_{n\times n}$ we get a worst-case complexity of $\bigO(n^2 \epsilon^{-2})$ iterations and $\bigO(n^3 \epsilon^{-2})$ objective evaluations.
However, using a Johnson-Lindenstrauss transform we may use $p=\bigO(1)$ and $P_{\max}=\bigO(\sqrt{n})$, giving the same iteration complexity but an improved complexity of $\bigO(n^2 \epsilon^{-2})$ evaluations.

\section{Convergence to Second-Order Stationary Points} \label{sec_second_order}
We now consider the extension of the framework in Section~\ref{sec_background} to a second-order convergent algorithm.
Our basic framework is the same, namely at each iteration we select a random $P_k$ and build a low-dimensional model $\hat{m}_k$ \eqref{eq_model_subspace} to approximate $f$ in the corresponding affine space $\mathcal{Y}_k$.

Now, we are interested in the second-order criticality measure
\begin{align}
    \sigma_k  := \max(\|\grad f(\bx_k)\|, \tau_k), \qquad\text{where}\qquad \tau_k := \max(-\lambda_{\min}(\grad^2 f(\bx_k)), 0)),
\end{align}
(see \cite[Chapter 10]{Conn2009}, for example).
Restricting the objective to the space $\mathcal{Y}_k$ we get the corresponding subspace criticality measure 
\begin{align}
    \hat{\sigma}_k  := \max(\|P_k^T\grad f(\bx_k)\|, \hat{\tau}_k), \qquad\text{where}\qquad \hat{\tau}_k := \max(-\lambda_{\min}(P_k^T \grad^2 f(\bx_k) P_k), 0)),
\end{align}
and our interpolation model $\hat{m}_k$ defines its own approximate (subspace) criticality measure
\begin{align}
    \hat{\sigma}^m_k  := \max(\|\hat{\bg}_k\|, \hat{\tau}^m_k), \qquad\text{where}\qquad \hat{\tau}^m_k := \max(-\lambda_{\min}(\hat{H}_k), 0)).
\end{align}
To achieve $\sigma_k\to 0$ (i.e.~convergence to second-order critical points), we need a more restrictive notion of model accuracy.


\begin{definition} \label{def_fully_quadratic}
    The subspace model $\hat{m}_k:\R^p\to\R$ in \eqref{eq_model_subspace} is $P_k$-fully quadratic if there exist constants $\kappaef,\kappaeg,\kappaeh$ independent of $k$ such that
    \begin{subequations}
    \begin{align}
        |f(\bx_k + P_k\hat{\bs}) - \hat{m}_k(\hat{\bs})| &\leq \kappaef \Delta_k^3, \\
        \|P_k^T \grad f(\bx_k+P_k\hat{\bs}) - \grad \hat{m}_k(\hat{\bs})\| &\leq \kappaeg \Delta_k^2, \\
        \|P_k^T \grad^2 f(\bx_k+P_k\hat{\bs}) P_k - \grad^2 \hat{m}_k(\hat{\bs})\| &\leq \kappaeh \Delta_k,
    \end{align}
    \end{subequations}
    for all $\|\hat{\bs}\| \leq \Delta_k$.
\end{definition}

This notion is discussed in \cite[Section 2.4]{Chen2024}, and a specific sampling mechanism for achieving this is described there.

\begin{algorithm}[tbh]
\begin{algorithmic}[1]
\Statex \textbf{Inputs:} initial iterate $\bx_0 \in \R^n$ and trust-region radius $\Delta_0>0$, and subspace dimension $p\in\{1,\ldots,n\}$.
\Statex \underline{Parameters:} maximum trust-region radius $\Delta_{\max} \geq \Delta_0$, trust-region parameters $0 < \gammadec < 1 < \gammainc$, and step acceptance thresholds $\eta\in(0,1)$ and $\mu>0$.
\vspace{0.2em}
\For{$k=0,1,2,\ldots$}
    \State Select a random matrix $P_k\in\R^{n\times p}$ satisfying Assumption~\ref{ass_prob_well_aligned}
    \State Build subspace model $\hat{m}_k$ \eqref{eq_model_subspace} which is $P_k$-fully quadratic (Definition~\ref{def_fully_quadratic}).
    \State Calculate a step $\hat{\bs}_k$ by approximately solving the trust-region subproblem \eqref{eq_trs}.
    \State Evaluate $f(\bx_k+P_k\hat{\bs}_k)$ and calculate the ratio $R_k$ \eqref{eq_ratio_test}.
    \If{$R_k \geq \eta$ and $\hat{\sigma}^m_k \geq \mu \Delta_k$}
        \State \textit{(successful)} Set $\bx_{k+1}=\bx_k+P_k\hat{\bs}_k$ and $\Delta_{k+1}=\min(\gammainc \Delta_k, \Delta_{\max})$.
    \Else
        \State \textit{(unsuccessful)} Set $\bx_{k+1}=\bx_k$ and $\Delta_{k+1}=\gammadec \Delta_k$.
    \EndIf
\EndFor
\end{algorithmic}
\caption{Random-subspace model-based DFO algorithm applied to  \eqref{eq_problem} with second-order convergence.}
\label{alg_rsdfo2}
\end{algorithm}

In addition, we need the following assumptions on the objective and computed trust-region step which are standard for second-order convergence (e.g.~\cite[Chapter 10]{Conn2009}).

\begin{assumption} \label{ass_smoothness_2}
    The objective $f$ is bounded below by $\flow$, and is twice continuously differentiable and $\grad^2 f$ is $L_H$-Lipschitz continuous.
    Additionally, the Hessian at all iterates is uniformly bounded, $\|\grad^2 f(\bx_k)\| \leq M$ for some $M>0$ independent of $k$.
\end{assumption}

\begin{assumption} \label{ass_cauchy_decrease_2}
    The step $\hat{\bs}_k$ satisfies
    \begin{align}
        \hat{m}_k(\bm{0}) - \hat{m}_k(\hat{\bs}_k) \geq \kappa_s \max\left(\|\hat{\bg}_k\| \min\left(\Delta_k, \frac{\|\hat{\bg}_k\|}{\max(\|\hat{H}_k\|,1)}\right), \hat{\tau}^m_k\Delta_k^2\right),
    \end{align}
    for some $\kappa_s>0$ independent of $k$.
\end{assumption}

We also require Assumption~\ref{ass_bounded_hess} (bounded model Hessians) holds here.

\begin{lemma}[Lemma 10.15, \cite{Conn2009}] \label{lem_fully_quadratic}
    Suppose Assumption~\ref{ass_smoothness_2} holds and $\hat{m}_k$ is $P_k$-fully quadratic on $B(\bx_k,\Delta_k)$.
    Then $|\hat{\sigma}_k - \hat{\sigma}^m_k| \leq \kappa_{\sigma} \Delta_k$, 
    where $\kappa_{\sigma} := \max(\kappaeg \Delta_{\max}, \kappaeh)$.

\end{lemma}

\begin{lemma} \label{lem_small_delta_successful}
    Suppose Assumptions~\ref{ass_bounded_hess} and \ref{ass_cauchy_decrease_2}  hold.
    If
    \begin{align}
        \Delta_k \leq c_0\, \hat{\sigma}^m_k, \qquad \text{where} \qquad c_0 := \min\left(\frac{1}{\mu}, \frac{1}{\kappa_H}, \frac{\kappa_s (1-\eta)}{\kappaef \Delta_{\max}}, \frac{\kappa_s (1-\eta)}{\kappaef}\right),
    \end{align}
    then iteration $k$ is successful (i.e.~$R_k \geq \eta$ and $\hat{\sigma}^m_k \geq \mu\Delta_k$).
\end{lemma}
\begin{proof}
    This argument is based on \cite[Lemma 10.17]{Conn2009}.
    First, since $c_0 \leq 1/\mu$, we immediately have $\hat{\sigma}^m_k \geq \mu\Delta_k$; it remains to show $R_k \geq \eta$.
    By definition of $\hat{\sigma}^m_k$, we either have $\hat{\sigma}^m_k = \|\hat{\bg}_k\|$ or $\hat{\sigma}^m_k = \hat{\tau}^m_k$.
    From Assumptions~\ref{ass_cauchy_decrease_2} and \ref{ass_bounded_hess}, we have
    \begin{align}
        \hat{m}_k(\bm{0}) - \hat{m}_k(\hat{\bs}_k) \geq \kappa_s \max\left(\|\hat{\bg}_k\| \min\left(\Delta_k, \frac{\|\hat{\bg}_k\|}{\kappa_H}\right), \hat{\tau}^m_k\Delta_k^2\right).
    \end{align}
    First, if $\hat{\sigma}^m_k = \|\hat{\bg}_k\|$, then $\Delta_k \leq \frac{1}{\kappa_H}\hat{\sigma}^m_k = \frac{1}{\kappa_H}\|\hat{\bg}_k\|$ implies
    \begin{align}
        \hat{m}_k(\bm{0}) - \hat{m}_k(\hat{\bs}_k) \geq \kappa_s \|\hat{\bg}_k\| \Delta_k = \kappa_s \hat{\sigma}^m_k \Delta_k. \label{eq_cauchy_case1}
    \end{align}
    Then we have
    \begin{align}
        |R_k-1| &= \frac{|f(\bx_k+P_k\hat{\bs}_k) - \hat{m}_k(\hat{\bs}_k)|}{|\hat{m}_k(\bm{0}) - \hat{m}_k(\hat{\bs}_k)|}, \\
        &\leq \frac{\kappaef \Delta_k^3}{\kappa_s \hat{\sigma}^m_k \Delta_k}, \\
        &\leq \frac{\kappaef \Delta_{\max}}{\kappa_s \hat{\sigma}^m_k} \Delta_k,
    \end{align}
    and so $|R_k-1| \leq 1-\eta$ (by definition of $c_0$), and hence $R_k \geq \eta$ as required.
    Alternatively, if $\hat{\sigma}^m_k = \hat{\tau}^m_k$, then
    \begin{align}
        \hat{m}_k(\bm{0}) - \hat{m}_k(\hat{\bs}_k) \geq \kappa_s \hat{\tau}^m_k\Delta_k^2 = \kappa_s \hat{\sigma}^m_k \Delta_k^2. \label{eq_cauchy_case2}
    \end{align}
    In this case, we get
    \begin{align}
        |R_k-1| &= \frac{|f(\bx_k+P_k\hat{\bs}_k) - \hat{m}_k(\hat{\bs}_k)|}{|\hat{m}_k(\bm{0}) - \hat{m}_k(\hat{\bs}_k)|} \leq \frac{\kappaef \Delta_k^3}{\kappa_s \hat{\sigma}^m_k \Delta_k^2} = \frac{\kappaef \Delta_k}{\kappa_s \hat{\sigma}^m_k},
    \end{align}
    and so again $|R_k-1| \leq 1-\eta$ by definition of $c_0$, and $R_k \geq \eta$.
\end{proof}

Now, we need to specify our requirements for a well-aligned subspace.
This is based on the approach in \cite{Cartis2024}, originally from the PhD thesis \cite[Chapter 5.6]{Shao2021}.
At iteration $k$, suppose $\grad^2 f(\bx_k)$ has eigendecomposition $\grad^2 f(\bx_k) = \sum_{i=1}^{r} \lambda_i \bv_i \bv_i^T$, where $\lambda_1 \geq \lambda_2 \geq \cdots \geq \lambda_r$ and $r := \operatorname{rank}(\grad^2 f(\bx_k))\leq n$, and $\bv_1,\ldots,\bv_r$ are orthonormal.
Define $\hat{\bv}_i := P_k^T \bv_i \in \R^p$, and so the (true) subspace Hessian is $P_k^T \grad^2 f(\bx_k) P_k = \sum_{i=1}^{r} \lambda_i \hat{\bv}_i \hat{\bv}_i^T$.

\begin{definition}\label{2:well-aligned}
    A subspace $P_k$ is well-aligned if
    \begin{subequations} \label{eq_well_aligned_2}
    \begin{align}
        \|P_k\| &\leq P_{\max}, \label{eq_well_aligned_2a} \\
        \|P_k^T \grad f(\bx_k)\| &\geq (1-\alpha) \|\grad f(\bx_k)\|, \label{eq_well_aligned_2b} \\
        \|\hat{\bv}_r\| &\geq 1-\alpha, \label{eq_well_aligned_2c} \\
        (\hat{\bv}_i^T \hat{\bv}_r)^2 &\leq 4\alpha^2, \qquad \forall i=1,\ldots,r-1, \label{eq_well_aligned_2d}
    \end{align}
    \end{subequations}
    for some $\alpha\in(0,1)$ and $P_{\max}>0$ independent of $k$.
\end{definition}

In practice, a well-aligned subspace can be achieved by generating $P_k$ randomly (see Section~\ref{sec_well_aligned} below).
Hence, we use the following assumption.

\begin{assumption} \label{ass_prob_well_aligned}
    At any iteration $k$, the matrix $P_k$ is well-aligned in the sense of Definition \ref{2:well-aligned} with probability at least $1-\delta_S$, for some $\delta_S\in[0,1)$, independent of all previous iterations $P_0,\ldots,P_{k-1}$.
\end{assumption}

The first two well-aligned properties \eqref{eq_well_aligned_2a} and \eqref{eq_well_aligned_2b} match the first-order case (Assumption~\ref{ass_well_aligned}); i.e.~the subspace sees a sizeable fraction of the true gradient.
The third condition \eqref{eq_well_aligned_2c} says the subspace sees a sizeable fraction of the left-most eigenvector (since $\|\bv_r\|=1$), and the last condition \eqref{eq_well_aligned_2d} says the first $r-1$ eigenvectors remain approximately orthogonal to the last eigenvector in the subspace.

\begin{lemma} \label{lem_subspace_sigma}
    Suppose Assumption~\ref{ass_smoothness_2} holds.
    If $P_k$ is well-aligned and $\sigma_k \geq \epsilon > 0$, and
    \begin{align}
        \theta := (1-\alpha)^2 - \frac{4M(r-1)\alpha^2}{\epsilon (1-\alpha)^2} > 0. \label{eq_theta_defn}
    \end{align}
    Then
    \begin{align}
        \hat{\sigma}_k \geq \min((1-\alpha)^2, \theta) \epsilon.
    \end{align}
\end{lemma}
\begin{proof}
    This argument is based on \cite[Lemma 6.10]{Cartis2024}, originally from \cite[Lemma 5.6.6]{Shao2021}.
    First, suppose that $\sigma_k=\|\grad f(\bx_k)\|$.
    Then
    \begin{align}
        \hat{\sigma}_k \geq \|P_k^T \grad f(\bx_k)\| \geq (1-\alpha) \|\grad f(\bx_k)\| = (1-\alpha) \sigma_k \geq (1-\alpha)^2 \epsilon,
    \end{align}
    using $(1-\alpha)^2 < 1-\alpha$ from $\alpha\in(0,1)$ and we are done.
    Instead, suppose that $\sigma_k = \tau_k$.
    Using Rayleigh quotients, we have
    \begin{align}
        \lambda_{\min}(P_k^T \grad^2 f(\bx_k) P_k) \leq \frac{\hat{\bv}_r^T (\sum_{i=1}^{r} \lambda_i \hat{\bv}_i \hat{\bv}_i^T) \hat{\bv}_r}{\|\hat{\bv}_r\|^2} = \lambda_r \|\hat{\bv}_r\|^2 + \frac{\sum_{i=1}^{r-1} \lambda_i (\hat{\bv}_i^T \hat{\bv}_r)^2}{\|\hat{\bv}_r\|^2}.
    \end{align}
    Since $\tau_k \geq \epsilon > 0$, we have $\tau_k = -\lambda_r$ and $\lambda_r< -\epsilon$, and so
    \begin{align}
        -\lambda_{\min}(P_k^T \grad^2 f(\bx_k) P_k) &\geq \epsilon \|\hat{\bv}_r\|^2 - \lambda_1 \frac{\sum_{i=1}^{r-1} (\hat{\bv}_i^T \hat{\bv}_r)^2}{\|\hat{\bv}_r\|^2}.
    \end{align}
    If $\lambda_1 \leq 0$ then the second term is non-negative, and so 
    \begin{align}
        -\lambda_{\min}(P_k^T \grad^2 f(\bx_k) P_k) \geq \epsilon (1-\alpha)^2 > 0,
    \end{align}
    and so $\hat{\tau}_k > 0$ with $\hat{\sigma}_k \geq \hat{\tau}_k \geq \epsilon (1-\alpha)^2$.
    Instead, if $\lambda_1 > 0$ then $\lambda_1 \leq \|\grad^2 f(\bx_k)\| \leq M$ and so we have
    \begin{align}
        -\lambda_{\min}(P_k^T \grad^2 f(\bx_k) P_k) &\geq \epsilon (1-\alpha)^2 - M \frac{4(r-1)\alpha^2}{(1-\alpha)^2} = \theta\epsilon > 0,
    \end{align}
    and so again $-\lambda_{\min}(P_k^T \grad^2 f(\bx_k) P_k) < 0$ yielding $\hat{\sigma}_k \geq \hat{\tau}_k > \theta\epsilon$ and the result follows.
\end{proof}

\begin{lemma} \label{lem_sigma_min}
    Suppose Assumption~\ref{ass_smoothness_2} holds and $\theta>0$ \eqref{eq_theta_defn}.
    If $P_k$ is well-aligned and $\sigma_k \geq \epsilon$, then if iteration $k$ is successful,
    \begin{align}
        \hat{\sigma}^m_k \geq c_1 \epsilon, \qquad \text{where} \qquad c_1 := \frac{\min(1-\alpha, (1-\alpha)^2, \theta)}{1+\kappa_{\sigma}\mu^{-1}}.
    \end{align}
\end{lemma}
\begin{proof}
    On successful iterations, we have $\hat{\sigma}^m_k \geq \mu \Delta_k$ and so from Lemmas~\ref{lem_fully_quadratic} and \ref{lem_subspace_sigma} we have
    \begin{align}
        \min(1-\alpha, (1-\alpha)^2, \theta)\epsilon \leq \hat{\sigma}_k \leq \hat{\sigma}^m_k + |\hat{\sigma}_k - \hat{\sigma}^m_k| \leq \hat{\sigma}^m_k + \kappa_{\sigma} \Delta_k \leq (1 + \kappa_{\sigma}\mu^{-1}) \hat{\sigma}^m_k,
    \end{align}
    and we are done.
\end{proof}

We are now ready to prove the worst-case complexity of Algorithm~\ref{alg_rsdfo2} to (approximate) second-order stationary points.
The below closely follows the argument from \cite{Cartis2023}, used to prove Theorem~\ref{thm_first_order_complexity}.
For a fixed $K$, define the following sets of iterations:
\begin{itemize}
    \item $\mathcal{A}$ is the set of iterations in $\{0,\ldots,K\}$ for which $P_k$ is well-aligned.
    \item $\mathcal{A}^C$ is the set of iterations in $\{0,\ldots,K\}$ for which $P_k$ is not well-aligned.
    \item $\mathcal{S}$ (resp.~$\mathcal{U}$) is the set of iterations in $\{0,\ldots,K\}$ which are successful (resp.~unsuccessful).
    \item $\mathcal{D}(\Delta)$ is the set of iterations in $\{0,\ldots,K\}$ for which $\Delta_k \geq \Delta$.
    \item $\mathcal{D}^C(\Delta)$ is the set of iterations in $\{0,\ldots,K\}$ for which $\Delta_k < \Delta$.
\end{itemize}
Hence, for any $K$ and $\Delta$, we have the partitions $\{0,\ldots,K\} = \mathcal{A}\cup \mathcal{A}^C = \mathcal{S} \cup \mathcal{U} = \mathcal{D}(\Delta) \cup \mathcal{D}^C(\Delta)$, as well as $\mathcal{D}(\Delta_1) \subseteq \mathcal{D}(\Delta_2)$ if $\Delta_1 \geq \Delta_2$.

We first use standard trust-region arguments to bound $\mathcal{A}$ (by considering the four cases $\mathcal{A} \cap \mathcal{D}(\Delta) \cap \mathcal{S}, \ldots, \mathcal{A} \cap \mathcal{D}^C(\Delta) \cap \mathcal{U}$ separately), provided $\sigma_k$ remains large.
We can then bound (to high probability) the total number of iterations $K$ for which $\sigma_k$ is large by noting that $\mathcal{A}^C$ occurs with small probability (Assumption~\ref{ass_prob_well_aligned}).

\begin{lemma} \label{lem_success_bound}
    Suppose Assumptions~\ref{ass_bounded_hess}, \ref{ass_smoothness_2} and \ref{ass_cauchy_decrease_2} hold, and $\theta>0$ \eqref{eq_theta_defn}.
    If $\sigma_k \geq \epsilon > 0$ for all $k\leq K$, then
    \begin{align}
        \#(\mathcal{A} \cap \mathcal{D}(\Delta) \cap \mathcal{S}) \leq \phi(\Delta,\epsilon), \qquad \text{where} \quad \phi(\Delta, \epsilon) := \frac{f(\bx_0) - \flow}{\eta \kappa_s c_1 \min(\epsilon \Delta, \epsilon \Delta^2)},
    \end{align}
    where $c_1>0$ is defined in Lemma~\ref{lem_sigma_min}.
\end{lemma}
\begin{proof}
    From \eqref{eq_cauchy_case1} and \eqref{eq_cauchy_case2}, we have 
    \begin{align}
        \hat{m}_k(\bm{0}) - \hat{m}_k(\hat{\bs}_k) \geq \min(\kappa_s \hat{\sigma}^m_k \Delta_k, \kappa_s \hat{\sigma}^m_k \Delta_k^2),
    \end{align}
    depending on whether $\hat{\sigma}^m_k = \|\hat{\bg}_k\|$ or $\hat{\sigma}^m_k = \hat{\tau}^m_k$.
    For $k\in \mathcal{A} \cap \mathcal{D}(\Delta)$, from Lemma~\ref{lem_sigma_min} we have $\Delta_k \geq \Delta$ and $\hat{\sigma}^m_k \geq c_1 \epsilon$ and so
    \begin{align}
        \hat{m}_k(\bm{0}) - \hat{m}_k(\hat{\bs}_k) \geq \kappa_s c_1 \min(\epsilon \Delta, \epsilon \Delta^2).
    \end{align}
    We note that $\bx_k$ is only changed for $k\in\mathcal{S}$, and so
    \begin{align}
        f(\bx_0) - \flow \geq \sum_{k\in\mathcal{S}} f(\bx_k) - f(\bx_k+P_k \hat{\bs}_k) \geq \sum_{k\in \mathcal{A} \cap \mathcal{D}(\Delta) \cap \mathcal{S}} \eta (\hat{m}_k(\bm{0}) - \hat{m}_k(\hat{\bs}_k)),
    \end{align}
    which yields
    \begin{align}
        f(\bx_0) - \flow \geq \eta \kappa_s c_1 \min(\epsilon \Delta, \epsilon \Delta^2) \cdot \#(\mathcal{A} \cap \mathcal{D}(\Delta) \cap \mathcal{S}),
    \end{align}
    which gives the result.
\end{proof}

\begin{lemma} \label{lem_unsuccess_bound}
    Suppose Assumptions~\ref{ass_bounded_hess}, \ref{ass_smoothness_2} and \ref{ass_cauchy_decrease_2}  hold, and $\theta>0$ \eqref{eq_theta_defn}.
    If $\sigma_k \geq \epsilon > 0$ for all $k\leq K$, then
    \begin{align}
        \#(\mathcal{A} \cap \mathcal{D}^C(\Delta) \cap \mathcal{U}) = 0,
    \end{align}
    for all $\Delta \leq c_2 \epsilon$, where
    \begin{align}
        c_2 := \frac{1}{\kappa_{\sigma} + c_0^{-1}} \min(1-\alpha, (1-\alpha)^2, \theta),
    \end{align}
    where $\kappa_{\sigma}>0$ is defined in Lemma~\ref{lem_fully_quadratic} and $c_0>0$ is defined in Lemma~\ref{lem_small_delta_successful}.
\end{lemma}
\begin{proof}


    This argument follows \cite[Lemma 4.2.3]{Garmanjani2015}.
    To find a contradiction, suppose $k\in \mathcal{A}\cap \mathcal{D}^C(\Delta) \cap \mathcal{U}$.
    Since $k\in\mathcal{A}$, we have $\hat{\sigma}_k \geq \min(1-\alpha, (1-\alpha)^2, \theta) \epsilon$ by Lemma~\ref{lem_subspace_sigma}.
    Also, since $k\in\mathcal{U}$, we must have $\Delta_k > c_0 \hat{\sigma}^m_k$ by Lemma~\ref{lem_small_delta_successful}.
    Thus,
    \begin{align}
        \min(1-\alpha, (1-\alpha)^2, \theta) \epsilon \leq \hat{\sigma}_k \leq (\hat{\sigma}_k - \hat{\sigma}^m_k) + \hat{\sigma}^m_k < \kappa_{\sigma} \Delta_k + c_0^{-1} \Delta_k,
    \end{align}
    where the last inequality follows from Lemma~\ref{lem_fully_quadratic}.
    By assumption on $\Delta$, this contradicts $k\in\mathcal{D}^C(\Delta)$, and we are done.
\end{proof}

\begin{lemma} \label{lem_trust_region_management_bounds}
    Suppose Assumptions~\ref{ass_bounded_hess},  \ref{ass_smoothness_2} and \ref{ass_cauchy_decrease_2}  hold, and $\theta>0$ \eqref{eq_theta_defn}.
    If $\sigma_k \geq \epsilon > 0$ for all $k\leq K$, then
    \begin{align}
        \#(\mathcal{D}(\gammadec^{-1} \Delta) \cap \mathcal{U}) &\leq c_3 \cdot \#(D(\gammainc^{-1} \Delta) \cap \mathcal{S}) + c_4(\Delta), \label{eq_delta_bound_1}
    \end{align}
    for all $\Delta\leq \Delta_0$, and 
    \begin{align}
        \#(\mathcal{D}^C(\gammainc^{-1} \Delta) \cap \mathcal{S}) &\leq c_3^{-1} \cdot \#(\mathcal{D}^C(\gammadec^{-1} \Delta) \cap \mathcal{U}), \label{eq_delta_bound_2}
    \end{align}
    for all $\Delta \leq \min(\Delta_0, \gammainc^{-1} \Delta_{\max})$, 
    where we define
    \begin{align}
        c_3 := \frac{\log(\gammainc)}{\log(\gammadec^{-1})}, \qquad \text{and} \qquad c_4(\Delta) := \frac{\log(\Delta_0/\Delta)}{\log(\gammadec^{-1})}.
    \end{align}
\end{lemma}
\begin{proof}
    
    This follows the same argument as in \cite[Lemmas 6 \& 7]{Cartis2023}, which we include here for completeness.
    In both cases, we consider the movement of $\log(\Delta_k)$, which is increased by at most $\log(\gammainc)$ if $k\in\mathcal{S}$ and decreased by $\log(\gammadec^{-1})$ if $k\in\mathcal{U}$.

    To show \eqref{eq_delta_bound_1}, we note:
    \begin{itemize}
        \item Since $\Delta \leq \Delta_0$, the value $\log(\Delta)$ is $\log(\Delta_0/\Delta)$ below the starting value $\log(\Delta_0)$.
        \item If $k\in\mathcal{S}$, then $\log(\Delta_k)$ increases by at most $\log(\gammainc)$, and $\Delta_{k+1} \geq \Delta$ is only possible if $\Delta_k \geq \gammainc^{-1} \Delta$.
        \item If $k\in\mathcal{U}$, then $\log(\Delta_k)$ is decreased by $\log(\gammadec^{-1})$ and $\Delta_{k+1} \geq \Delta$ if $\Delta_k \geq \gammadec^{-1} \Delta$.
    \end{itemize}
    Hence the total decrease in $\log(\Delta_k)$ from $k\in\mathcal{D}(\gammadec^{-1})\cap \mathcal{U}$ must be fully matched by the initial gap $\log(\Delta_0/\Delta)$ plus the maximum possible amount that $\log(\Delta_k)$ can be increased above $\log(\Delta)$; that is,
    \begin{align}
        \log(\gammadec^{-1}) \#(\mathcal{D}(\gammadec^{-1}\Delta)\cap \mathcal{S}) \leq \log(\Delta_0/\Delta) + \log(\gammainc) \#(\mathcal{D}(\gammainc^{-1} \Delta) \cap \mathcal{S}),
    \end{align}
    and we get \eqref{eq_delta_bound_1}.

    To get \eqref{eq_delta_bound_2}, we use a similar approach.
    Specifically,
    \begin{itemize}
        \item If $k\in\mathcal{S}$, then since $\Delta \leq \gammainc^{-1} \Delta_{\max}$, we have $\log(\Delta_k)$ is increased by exactly $\log(\gammainc)$, and $\Delta_{k+1} < \Delta$ if and only if $\Delta_k < \gammainc^{-1} \Delta$.
        \item If $k\in\mathcal{U}$, then $\log(\Delta_k)$ is decreased by $\log(\gammadec^{-1})$ and we need $\Delta_k < \gammadec^{-1} \Delta$ to get $\Delta_{k+1} < \Delta$.
    \end{itemize}
    Since $\Delta \leq \Delta_0$, the total increase in $\log(\Delta_k)$ from $k\in \mathcal{D}^C(\gammainc^{-1} \Delta) \cap \mathcal{S}$ must be fully matched by the total decrease from $k\in\mathcal{D}^C (\gammadec^{-1} \Delta)\cap \mathcal{U}$, and so
    \begin{align}
        \log(\gammainc) \#(\mathcal{D}^C (\gammainc^{-1} \Delta) \cap \mathcal{S}) \leq \log(\gammadec^{-1}) \#(\mathcal{D}^C(\gammadec^{-1} \Delta) \cap \mathcal{U}),
    \end{align}
    and we get \eqref{eq_delta_bound_2}.
\end{proof}

We can now bound the total number of well-aligned iterations.

\begin{lemma} \label{lem_aligned_bound}
    Suppose Assumptions~\ref{ass_bounded_hess}, \ref{ass_smoothness_2} and \ref{ass_cauchy_decrease_2}  hold, $\theta>0$ \eqref{eq_theta_defn}, and $\gammainc > \gammadec^{-1}$.
    If $\sigma_k \geq \epsilon > 0$ for all $k\leq K$, then
    \begin{align}
        \#(\mathcal{A}) \leq \psi(\epsilon) + \frac{c_5}{c_5+1} \cdot (K+1),
    \end{align}
    where
    \begin{align}
        \psi(\epsilon) &:= \frac{1}{c_5+1}\left[(c_3 + 1) \phi(\Delta_{\min}, \epsilon) + c_4(\gammainc\Delta_{\min}) + \frac{\phi(\gammainc^{-1} \gammadec \Delta_{\min}, \epsilon)}{1-c_3^{-1}}\right], \label{phi-eps}\\
        \Delta_{\min} &:= \min\left(\gammainc^{-1} \Delta_0, \gammadec^{-1} \gammainc^{-1} \Delta_{\max}, \gammadec \gammainc^{-1} c_2 \epsilon\right), \\
        c_5 &:= \max\left(c_3, \frac{c_3^{-1}}{1-c_3^{-1}}\right),
    \end{align}
    with $c_5>0$, where $c_3$ and $c_4(\Delta)$ are defined in Lemma~\ref{lem_trust_region_management_bounds}.
\end{lemma}
\begin{proof}
    This follows the proof of \cite[Lemma 9]{Cartis2023}.
    Since $\#(\mathcal{A}) = \#(\mathcal{A}\cap \mathcal{D}(\Delta_{\min})) + \#(\mathcal{A}\cap \mathcal{D}^C(\Delta_{\min}))$, we bound these terms individually.

    First, we use Lemma~\ref{lem_success_bound} to get
    \begin{align}
        \#(\mathcal{A}\cap \mathcal{D}(\Delta_{\min})) &= \#(\mathcal{A}\cap \mathcal{D}(\Delta_{\min})\cap\mathcal{S}) + \#(\mathcal{A}\cap \mathcal{D}(\Delta_{\min})\cap\mathcal{U}), \\
        &\leq \phi(\Delta_{\min}, \epsilon) + \#(\mathcal{A} \cap \mathcal{D}(\gammadec^{-1} \gammainc \Delta_{\min}) \cap \mathcal{U}) \nonumber \\
        &\qquad\qquad\quad\:\:\: + \#(\mathcal{A} \cap \mathcal{D}^C(\gammadec^{-1} \gammainc \Delta_{\min}) \cap \mathcal{D}(\Delta_{\min}) \cap \mathcal{U}), \\
        &\leq \phi(\Delta_{\min}, \epsilon) + c_3 \#(\mathcal{D}(\Delta_{\min}) \cap \mathcal{S}) + c_4(\gammainc\Delta_{\min}) + 0, \\
        &\leq (c_3 + 1) \phi(\Delta_{\min}, \epsilon) + c_3 \#(\mathcal{A}^C \cap \mathcal{D}(\Delta_{\min}) \cap \mathcal{S}) + c_4(\gammainc\Delta_{\min}),
    \end{align}
    where the second inequality follows from  Lemma~\ref{lem_trust_region_management_bounds} (specifically \eqref{eq_delta_bound_1} after noting $\gammainc \Delta_{\min} \leq \Delta_0$) and Lemma~\ref{lem_unsuccess_bound} (with $\gammadec^{-1} \gammainc \Delta_{\min} \leq c_2 \epsilon$), and the last inequality comes from using Lemma~\ref{lem_success_bound} a second time.

    Separately, we use Lemma~\ref{lem_unsuccess_bound} (with $\Delta_{\min} \leq c_2 \epsilon$) to get
    \begin{align}
        \#(\mathcal{A}\cap \mathcal{D}^C(\Delta_{\min})) &= \#(\mathcal{A} \cap \mathcal{D}^C(\Delta_{\min}) \cap \mathcal{S}), \\
        &= \#(\mathcal{A} \cap \mathcal{D}(\gammainc^{-1} \gammadec \Delta_{\min}) \cap \mathcal{D}^C (\Delta_{\min}) \cap \mathcal{S}) \nonumber \\
        &\qquad\qquad\qquad + \#(\mathcal{A} \cap \mathcal{D}^C(\gammainc^{-1} \gammadec \Delta_{\min}) \cap \mathcal{S}), \\
        &\leq \#(\mathcal{A} \cap \mathcal{D}(\gammainc^{-1} \gammadec \Delta_{\min}) \cap \mathcal{S}) + \#(\mathcal{D}^C(\gammainc^{-1} \gammadec \Delta_{\min}) \cap \mathcal{S}), \\
        &\leq \phi(\gammainc^{-1} \gammadec \Delta_{\min}, \epsilon) + c_3^{-1} \#(\mathcal{D}^C (\Delta_{\min}) \cap \mathcal{U}), \\
        &= \phi(\gammainc^{-1} \gammadec \Delta_{\min}, \epsilon) + c_3^{-1} \#(\mathcal{A} \cap \mathcal{D}^C (\Delta_{\min}) \cap \mathcal{U}) \nonumber \\
        &\qquad\qquad\qquad\qquad\qquad + c_3^{-1} \#(\mathcal{A}^C \cap \mathcal{D}^C (\Delta_{\min}) \cap \mathcal{U}), \label{eq_a_bound_1}
    \end{align}
    where the second inequality comes from Lemma~\ref{lem_success_bound} and Lemma~\ref{lem_trust_region_management_bounds} (specifically \eqref{eq_delta_bound_2} with $\gammadec \Delta_{\min} \leq \Delta_{\min} \leq \Delta_0$ and $\gammadec \Delta_{\min} \leq \gammainc^{-1} \Delta_{\max}$).
    We now use $\#(\mathcal{A} \cap \mathcal{D}^C (\Delta_{\min}) \cap \mathcal{U}) \leq \#(\mathcal{A} \cap \mathcal{D}^C (\Delta_{\min}))$ and rearrange the above to get
    \begin{align}
        \#(\mathcal{A}\cap \mathcal{D}^C(\Delta_{\min})) &\leq \frac{1}{1-c_3^{-1}} \left[\phi(\gammainc^{-1} \gammadec \Delta_{\min}, \epsilon) + c_3^{-1} \#(\mathcal{A}^C \cap \mathcal{D}^C (\Delta_{\min}) \cap \mathcal{U})\right], \label{eq_a_bound_2}
    \end{align}
    where the assumption $\gammainc > \gammadec^{-1} > 1$ guarantees $c_3^{-1} \in (0,1)$.

    We combine \eqref{eq_a_bound_1} and \eqref{eq_a_bound_2} to get
    \begin{align}
        \#(\mathcal{A}) \leq (c_3 + 1) \phi(\Delta_{\min}, \epsilon) + c_4(\gammainc\Delta_{\min}) + \frac{\phi(\gammainc^{-1} \gammadec \Delta_{\min}, \epsilon)}{1-c_3^{-1}}  \\
       + \max\left(c_3, \frac{c_3^{-1}}{1-c_3^{-1}}\right) \cdot \#(\mathcal{A}^C),
    \end{align}
    and the result then follows after using $\#(\mathcal{A}^C) = (K+1)-\#(\mathcal{A})$ and rearranging.
    To show that $c_5>0$, recall $c_3^{-1} \in (0,1)$.
\end{proof}

To get a total complexity bound, it suffices to note that $P_k$ is not well-aligned (i.e.~$\mathcal{A}^C$) for only a small fraction of all iterations, which comes from the following Chernoff-type bound.

\begin{lemma} \label{lem_bad_aligned_chernoff}
    Suppose Assumption~\ref{ass_prob_well_aligned} holds.
    Then
    \begin{align}
        \mathbb{P}\left[\#(\mathcal{A}) \leq (1-\delta)(1-\delta_S)(K+1)\right] \leq \exp\left(-\delta^2 (1-\delta_S) (K+1)/2\right),
    \end{align}
    for any $\delta\in(0,1)$.
\end{lemma}
\begin{proof}
    This is \cite[Eq.~(66)]{Cartis2023}. 
\end{proof}

Lastly, we follow the argument from \cite[Theorem 1]{Cartis2023} to get our main worst-case complexity bound.

\begin{theorem} \label{thm_wcc_2}
Suppose Assumptions~\ref{ass_bounded_hess}, \ref{ass_smoothness_2}, \ref{ass_cauchy_decrease_2} and \ref{ass_prob_well_aligned} hold, $\theta>0$ \eqref{eq_theta_defn}, $\gammainc > \gammadec^{-1}$, and $\delta_S < \frac{1}{1+c_5}$ for $c_5$ defined in Lemma~\ref{lem_aligned_bound}.
    Then for any $\epsilon>0$ and 
    \begin{align}
        k \geq \frac{2\psi(\epsilon)}{1-\delta_s - c_5/(c_5+1)} - 1, \label{eq_full_complexity_k},
    \end{align}
 where $\psi(\epsilon)$ is defined in \eqref{phi-eps},   we have
    \begin{align}
        \mathbb{P}\left[\min_{j\leq k} \sigma_j \leq \epsilon\right] \geq 1 - \exp\left(-(k+1) \cdot \frac{[1-\delta_S - c_5/(c_5+1)]^2}{8(1-\delta_S)}\right). \label{eq_full_complexity_k_bound}
    \end{align}
    Alternatively, if we define $K_{\epsilon} := \min\{k : \sigma_k \leq \epsilon\}$ we have
    \begin{align}
        \mathbb{P}\left[K_{\epsilon} \leq \left\lceil \frac{2\psi(\epsilon)}{1-\delta_s - c_5/(c_5+1)} - 1 \right\rceil\right] \geq 1 - \exp\left(-\psi(\epsilon) \cdot \frac{1-\delta_S - c_5/(c_5+1)}{4(1-\delta_S)}\right). \label{eq_full_complexity_eps_bound}
    \end{align}
\end{theorem}
\begin{proof}
    Fix some iteration $k$, and let $\epsilon_k := \min_{j\leq k} \sigma_j$ and $A_k$ be the number of well-aligned iterations in $\{0,\ldots,k\}$.
    If $\epsilon_k > 0$, from Lemma~\ref{lem_aligned_bound} we have
    \begin{align}
        A_k \leq \psi(\epsilon_k) + \frac{c_5}{c_5+1}(k+1).
    \end{align}
    If we choose $\delta\in(0,1)$ such that $(1-\delta)(1-\delta_S) > c_5/(c_5+1)$, then
    \begin{align}
        \mathbb{P}\left[\psi(\epsilon_k) \leq \left((1-\delta)(1-\delta_S) - \frac{c_5}{c_5+1}\right)(k+1)\right] &\leq \mathbb{P}\left[A_k \leq (1-\delta)(1-\delta_S)(k+1)\right], \\
        &\leq \exp\left(-\delta^2 (1-\delta_S) (k+1)/2\right), \label{eq_full_complexity_tmp1}
    \end{align}
    from Lemma~\ref{lem_bad_aligned_chernoff}.
    Now define
    \begin{align}
        \delta := \frac{1}{2}\left[1 - \frac{c_5}{(c_5+1)(1-\delta_S)}\right], \label{eq_full_complexity_delta_defn}
    \end{align}
    which gives
    \begin{align}
        (1-\delta)(1-\delta_S) = \frac{1}{2}\left[1-\delta_S + \frac{c_5}{c_5+1}\right] > \frac{c_5}{c_5+1},
    \end{align}
    using $1-\delta_S > c_5/(c_5+1)$ from our assumption on $\delta_S$.
    Then \eqref{eq_full_complexity_tmp1} gives
    \begin{align}
        \mathbb{P}\left[\psi(\epsilon_k) \leq \frac{1}{2}\left(1-\delta_S-\frac{c_5}{c_5+1}\right)(k+1)\right] &\leq \exp\left(-\delta^2 (1-\delta_S) (k+1)/2\right). \label{eq_full_complexity_tmp2}
    \end{align}
    If we have $\epsilon_k=0$, then since $\lim_{\epsilon\to 0^{+}} \psi(\epsilon)=\infty$, we again get \eqref{eq_full_complexity_tmp2} since the left-hand side of \eqref{eq_full_complexity_tmp2} is zero.

    Now, fix $\epsilon>0$ and choose $k$ satisfying \eqref{eq_full_complexity_k}.
    Since $\psi(\epsilon)$ is non-increasing in $\epsilon$, we have
    \begin{align}
        \mathbb{P}\left[\epsilon_k > \epsilon\right] &\leq \mathbb{P}\left[\psi(\epsilon_k) \leq \psi(\epsilon)\right], \\
        &\leq \mathbb{P}\left[\psi(\epsilon_k) \leq \frac{1}{2}\left(1-\delta_S-\frac{c_5}{c_5+1}\right)(k+1)\right], \\
        &\leq \exp\left(-\delta^2 (1-\delta_S) (k+1)/2\right)
    \end{align}
    where the second and third inequalities follow from \eqref{eq_full_complexity_k} and \eqref{eq_full_complexity_tmp2} respectively.
    We substitute the definition of $\delta$ \eqref{eq_full_complexity_delta_defn} to get our desired result \eqref{eq_full_complexity_k_bound}.

    Lastly, define
    \begin{align}
        k := \left\lceil \frac{2\psi(\epsilon)}{1-\delta_s - c_5/(c_5+1)} - 1 \right\rceil, \label{eq_full_complexity_tmp3}
    \end{align}
    satisfying \eqref{eq_full_complexity_k}.
    Then \eqref{eq_full_complexity_k_bound} gives
    \begin{align}
        \mathbb{P}[K_{\epsilon} > k] = \mathbb{P}[\epsilon_k > \epsilon] &\leq \exp\left(-(k+1) \cdot \frac{[1-\delta_S - c_5/(c_5+1)]^2}{8(1-\delta_S)}\right), \\
        &\leq \exp\left(-\psi(\epsilon) \cdot \frac{1-\delta_S - c_5/(c_5+1)}{4(1-\delta_S)}\right),
    \end{align}
    where we get the last inequality by substituting \eqref{eq_full_complexity_tmp3}, and we get \eqref{eq_full_complexity_eps_bound}.
\end{proof}

\paragraph*{Summary of complexity}
Our main result Theorem~\ref{thm_wcc_2} gives a worst-case complexity to achieve $\sigma_k \leq \epsilon$ of $O(\psi(\epsilon))$ iterations, with high probability (and where that probability converges to 1 as $\epsilon\to 0$).

For convenience, define $\kappa_d := \max(\kappaef, \kappaeg, \kappaeh, \kappa_H)$.
Then for $\epsilon$ sufficiently small\footnote{Specifically, we need $\epsilon \leq \min(\gammadec^{-1}\Delta_0, \gammadec^{-2}\Delta_{\max})/ c_2$ to get $\Delta_{\min}=\Theta(\epsilon)$ and $\epsilon \leq \gammainc / (\gammadec c_2)$ to ensure $\Delta_{\min}\leq 1$, which is required for $\phi(\Delta_{\min},\epsilon) = \Theta(\epsilon^{-1} \Delta_{\min}^{-2})$. We also need $\epsilon \leq 1$ in order to ignore the term $c_4(\gammainc \Delta_{\min}) \sim \log(\epsilon^{-1})$ in $\psi(\epsilon)$.} we have $\Delta_{\min} = \Theta(c_2 \epsilon) = \Theta(\kappa_d^{-1} \epsilon)$, and $\psi(\epsilon) = \Theta(c_1^{-1} \epsilon^{-1} \Delta_{\min}^{-2})$, with $c_1 = \Theta(\kappa_{\sigma}^{-1}) = \Theta(\kappa_d^{-1})$.

Hence, our overall (high probability) complexity bound is $O(\kappa_d^3 \epsilon^{-3})$ iterations.
If we assume $O(p^2)$ evaluations per iteration, which is required for fully quadratic models, we get an alternative (high probability) complexity of $O(p^2 \kappa_d^3 \epsilon^{-3})$ objective evaluations.
This matches the existing (full space) complexity bound for model-based DFO from \cite[Chapter 4.2]{Garmanjani2015}.

\begin{remark}
    Following the arguments from \cite[Section 2.5]{Cartis2023}, the above high-probability complexity bound can be used to derive alternative results, such as complexity bounds that hold in expectation, and almost-sure convergence of a subsequence of $\sigma_k$ to zero, for example.
\end{remark}

\subsection{Generating Well-Aligned Subspaces} \label{sec_well_aligned}
We now turn our attention to satisfying Assumption~\ref{ass_prob_well_aligned}, that is, the matrix $P_k$ is well-aligned with high probability.
We can achieve this by choosing $P_k$ to be a Johnson-Lindenstrauss Transform (JLT) (e.g.~\cite[Section 2.1]{Woodruff2014}).
This is a random matrix ensemble that approximately preserves norms of finite collections of vectors, the most common example being where the entries of $P_k$ are chosen to be i.i.d.~$N(0,p^{-1})$ distributed, although other constructions exist.
See \cite[Chapter 4]{Shao2021} and \cite[Section 2.6]{Cartis2023} for more details.

To satisfy Assumption~\ref{ass_prob_well_aligned}, we choose $P_k$ to be a JLT such that (with high probability) (based on ideas from \cite[Chapter 5]{Shao2021} and \cite{Cartis2024})
\begin{subequations}
\begin{align}
    (1-\alpha)\|\grad f(\bx_k)\| &\leq \|P_k^T \grad f(\bx_k)\| \leq (1+\alpha) \|\grad f(\bx_k)\|, \\
    (1-\alpha)\|\bv_r\| &\leq \|\hat{\bv}_r\| \leq (1+\alpha)\|\bv_r\|, \\
    (1-\alpha)\|\bv_i+\bv_r\| &\leq \|\hat{\bv}_i+\hat{\bv}_r\| \leq (1+\alpha)\|\bv_i+\bv_r\| , \qquad \forall i=1,\ldots,r-1, \\
    (1-\alpha)\|\bv_i-\bv_r\| &\leq \|\hat{\bv}_i-\hat{\bv}_r\| \leq (1+\alpha)\|\bv_i-\bv_r\| , \qquad \forall i=1,\ldots,r-1,
\end{align}
\end{subequations}
all hold, where $r := \operatorname{rank}(\grad^2 f(\bx_k)) \leq n$.
This requires (approximately) preserving the norm of $2r$ vectors, and so can be achieved with subspace dimension $p=O(\alpha^{-2}\log r)$.
The additional requirement $\|P_k\| \leq P_{\max}$ can also be achieved with high probability without any modification to $p$, with $P_{\max}=\Theta(\sqrt{n/p})$ typical, e.g.~\cite[Table 2]{Roberts2023}.

The above immediately gives the first three conditions for a well-aligned subspace, \eqref{eq_well_aligned_2a}, \eqref{eq_well_aligned_2b} and \eqref{eq_well_aligned_2c}.
For the last condition \eqref{eq_well_aligned_2d}, we use the polarization identity\footnote{That is, $\bx^T \by = \frac{1}{4}(\|\bx+\by\|^2 - \|\bx-\by\|^2)$ for any $\bx,\by\in\R^n$.} to compute
\begin{align}
    \hat{\bv}_i^T \hat{\bv}_r &= \frac{1}{4}\left(\|\hat{\bv}_i+\hat{\bv}_r\|^2 - \|\hat{\bv}_i-\hat{\bv}_r\|^2\right), \\
    &\leq \frac{1}{4}\left((1+\alpha)^2\|\bv_i+\bv_r\|^2 - (1-\alpha)^2\|\bv_i-\bv_r\|^2\right), \\
    &= (1+\alpha^2)\bv_i^T \bv_r + \frac{1}{4}\left(2\alpha\|\bv_i+\bv_r\|^2 + 2\alpha\|\bv_i-\bv_r\|^2\right), \\
    &= 2\alpha,
\end{align}
where the last equality uses the fact that $\bv_i$ and $\bv_r$ are orthonormal.
Similar reasoning yields $\hat{\bv}_i^T \hat{\bv}_r \geq -2\alpha$ and so \eqref{eq_well_aligned_2d} is achieved.

In terms of the subspace fully quadratic constants, we note that $\hat{f}_k(\bs) := f(\bx_k+P_k\bs)$ has $(P_{\max}^3 L_H)$-Lipschitz continuous Hessian under Assumption~\ref{ass_smoothness_2}.
Recalling the definition $\kappa_d:=\max(\kappaef,\kappaeg,\kappaeh)$, for models built using $\Lambda$-poised fully quadratic interpolation, we have $\kappa_d=O(p_1^{3/2} L \Lambda)$ \cite[Theorems 3.14 \& 3.16]{Conn2009}, where $p_1$ is the number of interpolation points needed.
So, using full space methods we have the standard values $\kappa_d=O(n^3 L_H \Lambda)$.

For our subspace models, we have $p_1=O(p^2/2)$ and $L=O(P_{\max}^3 L_H)$, and so we have $\kappa_d=O(p^3 P_{\max}^3 L_H \Lambda)$.
So, using JLTs as above with $p=O(\log r)$ and $P_{\max}=O(\sqrt{n/p})$ we get $\kappa_d=O(p^{3/2} n^{3/2} L_H \Lambda) = \tilde{O}(n^{3/2} L_H \Lambda)$.

That is, full space methods use $p=n$ with $\kappa_d=O(n^3)$ and subspace methods require $p=O(\log r) \leq O(\log n)$ with $\kappa_d=\tilde{O}(n^{3/2})$, a significant improvement for large $n$.
Specifically, using a complexity of $O(\kappa_d^5 \epsilon^{-3})$ from above, the dependency on $n$ decreases from $O(n^{9})$ iterations for full-space methods to $\tilde{O}(n^{4.5})$ for subspace methods.
This is summarized in Table~\ref{tab_constants_comparison}.

\begin{table}[tbh]
    \begin{center}
        \begin{tabular}{cccccc}
            \hline
            Method & $p$ & $P_{\max}$ & $\kappa_d$ & Iteration WCC & Evaluation WCC \\ \hline
            Full space & $n$ & $1$ & $O(n^3)$ & $O(n^{9} \epsilon^{-3})$ & $O(n^{11} \epsilon^{-3})$ \\
            Subspace & $O(\log r)$ & $\Theta(\sqrt{n/p})$ & $\tilde{O}(n^{1.5})$ & $\tilde{O}(n^{4.5}\epsilon^{-3})$ & $\tilde{O}(n^{4.5}\epsilon^{-3})$ \\
            \hline 
        \end{tabular}
        \caption{Summary of relevant constants for full space and subspace methods, based on $\kappa_d = O(p^3 P_{\max}^3)$, iteration complexity $O(\kappa_d^3 \epsilon^{-3})$ and evaluation complexity $O(p^2 \kappa_d^3 \epsilon^{-3})$ (WCC = worst-case complexity).}
        \label{tab_constants_comparison}
    \end{center}
\end{table}

\begin{remark}
    A limitation of the above theory is the requirement $\theta > 0$ \eqref{eq_theta_defn}.
    If we pick $\alpha\ll 1$ and assume $M=O(1)$, then we get $\alpha^2 = O(\epsilon/r)$.
    To apply the JLT, we actually need $p \geq O(\alpha^{-2}) = O(r \epsilon^{-1})$, which gives $p=O(\epsilon^{-1} r \log r)$.
    Hence for $\theta > 0$ to hold, we need to seek low accuracy solutions, or have a situation where $r\ll n$, i.e.~the problem has \emph{low effective rank}.
    This structure does arise in practice, see \cite[Section 1]{Cartis2022}, \LR{ \cite[Section 1]{Cartis-Massart2024}} and references therein for examples, with a notable example being finetuning of large language models, where scalable DFO methods are practically useful \cite{Malladi2023}.

    In the case where $P_k$ is generated using scaled Gaussians (i.e.~each entry of $P_k$ is i.i.d.~$N(0,p^{-1})$), then this can be improved (see \cite[Lemma 6.6]{Cartis2024} or \cite[Lemma 5.6.2]{Shao2021}).
    Then, we compute
    \begin{align}
        (\hat{\bv}_i^T \hat{\bv}_r)^2 = \|\hat{\bv}_r\|^2 \left(\hat{\bv}_i^T \left(\frac{\hat{\bv}_r}{\|\hat{\bv}_r\|}\right)\right)^2 \leq (1+\alpha)^2 O(p^{-1}),
    \end{align}
    since $\hat{\bv}_i^T \by$ is a $N(0,p^{-1})$-distributed random variable for any unit vector $\by$.
    In this case, we can get a different expression for $\theta$, namely
    \begin{align}
        \theta = (1-\alpha)^2 - \frac{M(r-1)(1+\alpha)^2 O(p^{-1})}{\epsilon (1-\alpha)^2},
    \end{align}
    and so for $\theta>0$ we still require either low accuracy solutions or low effective rank, via $\epsilon > O(r p^{-1})$, but without the stringent requirement on $\alpha$ (and hence on $p$).
\end{remark}

\section{Practical Implementation} \label{sec_implementation}
In \cite[Section 4]{Cartis2023} the algorithm DFBGN, a practical implementation of Algorithm~\ref{alg_rsdfo} for nonlinear least-squares problems was presented, based on linear interpolation models (for each residual in the least-squares sum).
DFBGN was designed to efficiently use objective evaluations, a key principle of DFO methods (which are often used when objective evaluations are slow/expensive to obtain).
In this section, we outline a RSDFO-Q (random subspace DFO with quadratic models), a practical implementation of the ideas in Algorithms~\ref{alg_rsdfo} and \ref{alg_rsdfo2}, for general objectives \eqref{eq_problem}, based on quadratic interpolation models. 

At the core of DFBGN is the maintenance of an interpolation set comprising $p+1$ points including the current iterate, say $\{\bx_k, \bx_k+\bd_1, \ldots, \bx_k+\bd_p\}$ for some linearly independent set $\bd_1,\ldots,\bd_p\in\R^n$.
Points that are added to this set are either from a trust-region step \eqref{eq_trs} or of the form $\bx_k+\Delta_k \bd$ for random (mutually orthogonal) unit vectors $\bd$, orthogonal to all current $\bd_i$ (which is optimal in a specific sense, see \cite[Theorem 4]{Chen2024}).
Points are removed from this set using geometric considerations to ensure the full linearity of the resulting model, following the approach in \cite{Powell2009,Cartis2019}.
Given this interpolation set, the subspace $P_k$ in Algorithm~\ref{alg_rsdfo2} is effectively defined by $P_k=[\bd_1 \: \cdots \: \bd_p]$.

For a general objective problem \eqref{eq_problem}, quadratic interpolation models, with at least $p+2$ interpolation points (including $\bx_k$), are typically preferred to linear interpolation models.
However, in this situation we lose the important feature of DFBGN, namely that \emph{every interpolation point (except $\bx_k$) defines a direction in the subspace}.
In this section, we give a practical approach to avoiding this issue, and provide a subspace method for solving the general problem \eqref{eq_problem} based on (underdetermined) quadratic interpolation models.

As with DFBGN, the design of RSDFO-Q is based on the principles:

\begin{itemize}
    \item Low linear algebra cost: the per-iteration linear algebra cost should be linear in $n$ (but may be higher in $p$, since the algorithm is designed for $p\ll n$)
    \item Efficient use of objective evaluations: objective evaluations should be re-used where possible, while restricting exploration at each iteration to a $p$-dimensional subspace. In particular, the algorithm with $p=n$ should have broadly comparable performance to state-of-the-art (full-space) DFO methods.\footnote{Given the added flexibility of the subspace method to set $p\ll n$, and their resulting ability to handle problems with much larger $n$, it is unlikely that a subspace method with $p=n$ will match/outperform state-of-the-art full-space methods.}
\end{itemize}

The key idea to extend the DFBGN approach to quadratic models is to maintain two sets of interpolation points.
To work in a $p$-dimensional subspace with at most $p+2 \leq q \leq (p+1)(p+2)/2$ interpolation points, we maintain a `primary' set of $p+1$ interpolation points $\mathcal{Y}_1=\{\bx_k,\bx_k+\bd_1,\ldots,\bx_k+\bd_p\}$, which define the current subspace $P_k=[\bd_1 \: \cdots \: \bd_p]$ just as in DFBGN.
We separately maintain a `secondary' set $\mathcal{Y}_2$ of at most $q-p-1$ interpolation points, which is disjoint from the primary set.
Points in the secondary set are only used for model construction, whereas points in the primary set are used for model construction as well as defining the subspace $P_k$ and geometry management.
Whenever points are removed from the primary set, which is based on similar principles to DFBGN, those points are added to the secondary set (and the oldest points from the secondary set are deleted).
We note that the case $q=(p+1)(p+2)/2$ in general can yield fully quadratic (subspace) models, as required for Algorithm~\ref{alg_rsdfo2}, and smaller values of $q$ can yield fully linear (subspace) models, as required for Algorithm~\ref{alg_rsdfo}; see \cite[Chapters 3\& 5]{Conn2009} for details.

Given the subspace defined by $P_k$, we work with an orthonormal basis $Q_k$ for the same space for convenience.
By construction of $P_k$, all the primary interpolation points are in the subspace, but the secondary points likely will not be.
Hence, before building an interpolation model, we project all secondary interpolation points into the subspace via
\begin{align}
    \hat{\bs} = Q_k Q_k^T (\by-\bx_k), \qquad \forall \by\in \mathcal{Y}_2.
\end{align}
For notational convenience, we have $\hat{\bs}=\by-\bx_k$ for all $\by\in\mathcal{Y}_1$.
Given these perturbed points, the subspace model construction for \eqref{eq_model_subspace} is based on minimum Frobenius norm quadratic interpolation \cite{Powell2004}.
That is, we solve the equality-constrained QP
\begin{align}
    \min_{\hat{\bg}, \hat{H}} \|\hat{H} - \widetilde{H}_{k-1}\|_F^2 \qquad \text{s.t.} \quad \hat{H}=\hat{H}^T, \quad \hat{m}(\hat{\bs})=f(\by), \: \forall \by\in \mathcal{Y}_1\cup\mathcal{Y}_2, \label{eq_min_frob_subspace}
\end{align}
where for any $\by\in \mathcal{Y}_1\cup\mathcal{Y}_2$ we have the associated $\hat{\bs}$ defined above, and $\widetilde{H}_{k-1} := Q_k^T Q_{k-1} \hat{H}_{k-1} Q_{k-1}^T Q_k$ is the projection of $\hat{H}_{k-1}$ into the current subspace (with $\hat{H}_{-1}=0$).

\begin{algorithm}[H]
\begin{algorithmic}[1]
\Statex \textbf{Inputs:} initial iterate $\bx_0 \in \R^n$, subspace dimension $p\in\{1,\ldots,n\}$, maximum number of interpolation points $p+2 \leq q \leq (p+1)(p+2)/2$.
\Statex \underline{Parameters:} initial trust-region radius $\Delta_0>0$, maximum trust-region radius $\Delta_{\max}$, trust-region radii thresholds $0 < \gamma_S,\gammadec < 1 < \gammainc \leq \overgammainc$ and $0 < \alpha_1 \leq \alpha_2 < 1$ and acceptance thresholds $0 < \eta_1 \leq \eta_2 < 1$, and minimum steps between $\rho_k$ reductions $N\in\N$.
\vspace{0.2em}
\State Define the primary interpolation set $\mathcal{Y}_1 := \{\bx_0,\bx_0+\Delta_0 \bd_1\ldots,\bx_0+\Delta_0 \bd_p\}$ for $p$ random orthonormal vectors $\bd_1,\ldots,\bd_p$, and secondary interpolation set $\mathcal{Y}_2 = \emptyset$.
\State Set $H_{-1}=0\in\R^{p\times p}$, $Q_{-1}\in\R^{n\times p}$ arbitrary and $\rho_0=\Delta_0$.
\For{$k=0,1,2,\ldots$}
    \State Build $Q_k\in\R^{n\times p}$ with orthonormal columns forming a basis for $\{\by-\bx_k : \by \in \mathcal{Y}_1\setminus\{\bx_k\}\}$.
    \State Build subspace quadratic model $\hat{m}_k$ \eqref{eq_model_subspace} using minimum Frobenius norm quadratic interpolation \eqref{eq_min_frob_subspace}.
    \State Calculate a step $\hat{\bs}_k$ by approximately solving the trust-region subproblem \eqref{eq_trs}.
    \State Set \texttt{CAN\_REDUCE\_RHO=TRUE} if $k\geq N$, $\rho_{k-N}=\rho_k$ and $\min(\|\hat{\bs}_j\|,\Delta_j) \leq \rho_j$ for all $j=k-N,\ldots,k$ and \texttt{FALSE} otherwise.
    \If{$\|\hat{\bs}_k\| < \gamma_S \rho_k$}
        \State (\textit{safety}) Set $R_k=-1$ and $\Delta_{k+1} = \max(\gammadec \Delta_k, \rho_k)$ and $\bx_{k+1}=\bx_k$.
        \State If \texttt{CAN\_REDUCE\_RHO=FALSE} or $\Delta_k > \rho_k$, then remove 1 point from $\mathcal{Y}_1$ using Algorithm~\ref{alg_remove_single_point}. \label{ln_safety_remove}
    \Else
        \State Evaluate $f(\bx_k+Q_k\hat{\bs}_k)$ and calculate the ratio $R_k$ \eqref{eq_ratio_test}.
        \State Update trust-region radius
        \begin{align}
            \Delta_{k+1} = \begin{cases} \max(\min(\gammadec \Delta_k, \|\hat{\bs}_k\|), \rho_k), & R_k < \eta_1, \\ \max(\gammadec \Delta_k, \|\hat{\bs}_k\|, \rho_k), & \eta_1 \leq R_k  \leq \eta_2, \\ \min(\max(\gammainc \Delta_k, \overgammainc \|\hat{\bs}_k\|), \Delta_{\max}), & R_k > \eta_2. \end{cases}
        \end{align}  
        \State Set $\bx_{k+1} = \bx_k+Q_k\hat{\bs}_k$ if $R_k > 0$ and $\bx_{k+1}=\bx_k$ otherwise.
        \If{$p<n$}
            \State Set $\mathcal{Y}_1=\mathcal{Y}_1 \cup \{\bx_k + Q_k\hat{\bs}_k\}$.
            \State Move $\min(\max(\pdrop,2),p)$ points from $\mathcal{Y}_1$ to $\mathcal{Y}_2$ with Algorithm~\ref{alg_remove_single_point}. \label{ln_pdrop1} 
        \Else
            \State Move one point from $\mathcal{Y}_1$ to $\mathcal{Y}_2$ using Algorithm~\ref{alg_remove_single_point} and set $\mathcal{Y}_1 = \mathcal{Y}_1 \cup \{\bx_k+Q_k\hat{\bs}_k\}$. \label{ln_full_set_remove} 
            \State Move $\min(\max(\pdrop,1),p)$ points from $\mathcal{Y}_1$ to $\mathcal{Y}_2$ with Algorithm~\ref{alg_remove_single_point}. \label{ln_pdrop2}
        \EndIf
    \EndIf
    \State (\textit{reduce $\rho_k$}) If $R_k < 0$, $\Delta_k \leq \rho_k$, and \texttt{CAN\_REDUCE\_RHO=TRUE}, then set $\rho_{k+1}=\alpha_1 \rho_k$ and $\Delta_{k+1}=\alpha_2 \rho_k$. Otherwise, set $\rho_{k+1}=\rho_k$.  
    \State (\textit{add new points}) Let $\padd := p + 1 - |\mathcal{Y}_1|$ and generate random orthonormal vectors $\{\bd_1,\ldots,\bd_{\padd}\}$ that are also orthogonal to $\{\by-\bx_{k+1} : \by\in \mathcal{Y}_1\setminus\{\bx_{k+1}\}\}$.
    \State Set $\mathcal{Y}_1 = \mathcal{Y}_1 \cup \{\bx_{k+1} + \Delta_{k+1}\bd_1, \ldots, \bx_{k+1} + \Delta_{k+1}\bd_{\padd}\}$, and let $\bx_{k+1}$ be the point in $\mathcal{Y}_1$ with smallest objective value.
\EndFor
\end{algorithmic}
\caption{RSDFO-Q (Random subspace DFO with quadratic models) for \eqref{eq_problem}.}
\label{alg_subspace_bobyqa}
\end{algorithm}
\clearpage

A complete description of RSDFO-Q is given in Algorithm~\ref{alg_subspace_bobyqa}.
Aside from the more complex model construction in \eqref{eq_min_frob_subspace}, it closely follows the overall structure of DFBGN \cite{Cartis2023}.
The trust-region updating mechanism is more complex than DFBGN, as it employs a lower bound $\rho_k \leq \Delta_k$, where $\rho_k$ is only periodically reduced, and does not evaluate the objective at the tentative trust-region step if the step is too short.
Both these complications are drawn from BOBYQA \cite{Powell2009}.

There are two situations where points are to be removed from the primary interpolation set $\mathcal{Y}_1$.
The first is where a single point is removed (lines~\ref{ln_safety_remove} and \ref{ln_full_set_remove}), and the second is where potentially multiple points are removed (lines~\ref{ln_pdrop1} and \ref{ln_pdrop2}).
In the case where a single point is removed, Algorithm~\ref{alg_remove_single_point} gives details on how this is done.
In the case where multiple points are removed, \cite[Algorithm 4]{Cartis2023} is used to select the points to remove.
These two approaches are very similar, differing only in whether the tentative trust-region step is used in the Lagrange polynomial evaluation: they try to remove points which are far from the current iterate and/or degrade the geometry of the primary interpolation set.
In both cases, the points removed from $\mathcal{Y}_1$ are added to $\mathcal{Y}_2$ (with the oldest points removed from $\mathcal{Y}_2$ if it grows too large).
Following the heuristic for DFBGN \cite[Section 4.5]{Cartis2023}, the number of points to remove in lines~\ref{ln_pdrop1} and \ref{ln_pdrop2} is $p/10$ for unsuccessful steps ($R_k<0$) and 1 otherwise.

\begin{remark}
    The default parameter choices in Algorithm~\ref{alg_subspace_bobyqa} are $\Delta_0=0.1\max(\|\bx_0\|_{\infty},1)$, $\Delta_{\max}=10^{10}$, $\gamma_S=\gammadec=0.5$, $\gammainc=2$, $\overgammainc=4$, $\alpha_1=0.1$, $\alpha_2=0.5$, $\eta_1=0.1$, $\eta_2=0.7$, and $N=5$.
\end{remark}

\begin{algorithm}[tb]
\begin{algorithmic}[1]
\Statex \textbf{Inputs:} primary interpolation set $\mathcal{Y}_1$ with $p+1$ points including the current iterate $\bx_k$, secondary interpolation set $\mathcal{Y}_2$ with at most $q-p-1$ points, tentative trust-region step $Q_k\hat{\bs}_k$, trust-region radius $\Delta_k>0$
\vspace{0.2em}
\State Compute the linear Lagrange polynomials for all points in $\mathcal{Y}_1=\{\bx_{k+1},\by_1,\ldots,\by_p\}$
\State Let $\by_t\in\mathcal{Y}_1\setminus\{\bx_{k+1}\}$ be the point that maximizes
\begin{align}
    \theta_t = |\ell_t(\bx_k+Q_k\hat{\bs}_k)| \cdot \max\left(\frac{\|\by_t-\bx_k\|^4}{\Delta_k^4}, 1\right).
\end{align}
\State Set $\mathcal{Y}_1=\mathcal{Y}_1\setminus\{\by_t\}$ and $\mathcal{Y}_2=\mathcal{Y}_2\cup\{\by_t\}$
\State If $|\mathcal{Y}_2| > q-p-1$, remove the oldest point from $\mathcal{Y}_2$.
\end{algorithmic}
\caption{Remove a single point from $\mathcal{Y}_1$.}
\label{alg_remove_single_point}
\end{algorithm}

\section{Numerical Results} \label{sec_numerics}

We now demonstrate the numerical performance of RSDFO-Q in light of the underlying principles described in Section~\ref{sec_implementation}. 
In particular, we benchmark against Py-BOBYQA \cite{Cartis2019b} as a state-of-the-art full space DFO method based on the same minimum Frobenius norm quadratic interpolation mechanism.
It also shares many of the complexities in the trust-region updating mechanism that are present in Algorithm~\ref{alg_subspace_bobyqa}.

\subsection{Testing Framework}
We test our solvers on two collections of problems from the CUTEst set~\cite{Gould2015,Fowkes2022}:
\begin{itemize}
    \item The medium-scale collection `CFMR' of 90 problems with $25 \leq n \leq 120$ (with $n=100$ most common) from \cite[Section 7]{Cartis2019b}.
    \item The large-scale collection `CR-large' of 28 problems with $1000 \leq n \leq 5000$ (with $n=1000$ most common) from \cite[Appendix C]{Cartis2023}.
\end{itemize}
We terminate all solvers after $100(n+1)$ objective evaluations and 12 hours (wall) runtime, whichever occurs first\footnote{All solvers are implemented in Python with key linear algebra routines implemented in NumPy. The objective evaluation time for all problems was negligible.}, and run each randomized solver 10 times on each problem (from the same starting point) to get a good assessment of the average-case behavior.
All solvers had the same starting point and initial trust-region radius.

For each problem instance and solver, we measure the total objective evaluations required to find a point $\bx$ with $f(\bx) \leq f_{\min} + \tau(f(\bx_0) - f_{\min})$, where $f_{\min}$ is the true minimum listed in \cite{Cartis2019b,Cartis2023} and $\tau\ll 1$ is an accuracy measure (smaller values mean higher accuracy solutions).
In our results, we show accuracy levels $\tau\in\{10^{-1},10^{-3}\}$, as it was observed in \cite[Section 5]{Cartis2023} that subspace methods are most suited to finding lower-accuracy solutions.
If a solver terminates on budget or runtime before a sufficiently accurate point $\bx$ was found, the total objective evaluations is considered to be $\infty$.

Solvers are compared using data profiles \cite{More2009}, showing the proportion of problem instances solved within a given number of evaluations in units of $(n+1)$ for problems in $\R^n$ (i.e.~the cost of estimating a full gradient using finite differencing), and performance profiles \cite{Dolan2002}, which plot the proportion of problem instances solved within a given ratio of the minimum evaluations required by any instance of any solver.

\subsection{Medium-Scale Results}
We first show results for medium-scale problems based on the `CFMR' problem collection ($n\approx 100$).\footnote{The 12hr runtime limit was not reached for any of these problem instances.}
In Figure~\ref{fig_cutest100}, we compare Py-BOBYQA with $q\in\{n+1,n+2,2n+1\}$ interpolation points (i.e.~linear or underdetermined quadratic interpolation models) against RSDFO-Q with $q=2p+1$ interpolation points, with subspace dimension $p\in\{n/10, n/4, n/2, n\}$.

We see that the full-space method Py-BOBYQA is the best-performing algorithm, with larger numbers of interpolation points $q$ giving the best performance.
Among the subspace methods, those with larger choices of $p$ achieve the best performance, aligning with the results in \cite{Cartis2023}.
In particular (c.f.~principles in Section~\ref{sec_implementation}), we see that RSDFO-Q with full blocks $p=n$ (and $q=2p+1$ interpolation points) achieves broadly comparable performance to Py-BOBYQA (worse than Py-BOBYQA with $q=2n+1$ points but better than $q=n+2$ points), while having the significant additional flexibility of working in a subspace of arbitrary dimension $p$.


\begin{figure}[tbh]
\centering
\subfloat[Data profile, $\tau=10^{-1}$.]{%
\resizebox*{7cm}{!}{\includegraphics{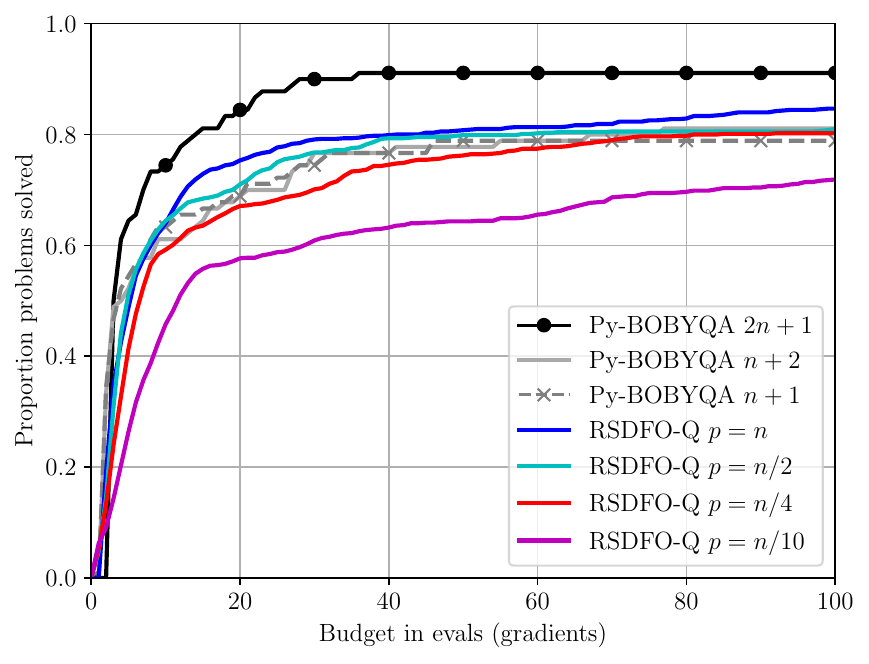}}}\hspace{5pt}
\subfloat[Performance profile, $\tau=10^{-1}$.]{%
\resizebox*{7cm}{!}{\includegraphics{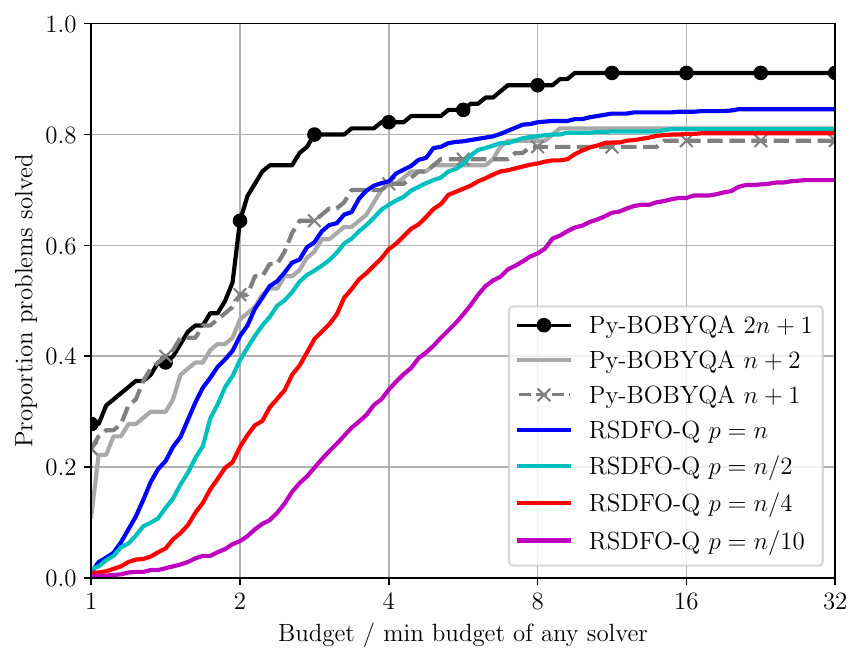}}}
\\
\subfloat[Data profile, $\tau=10^{-3}$.]{%
\resizebox*{7cm}{!}{\includegraphics{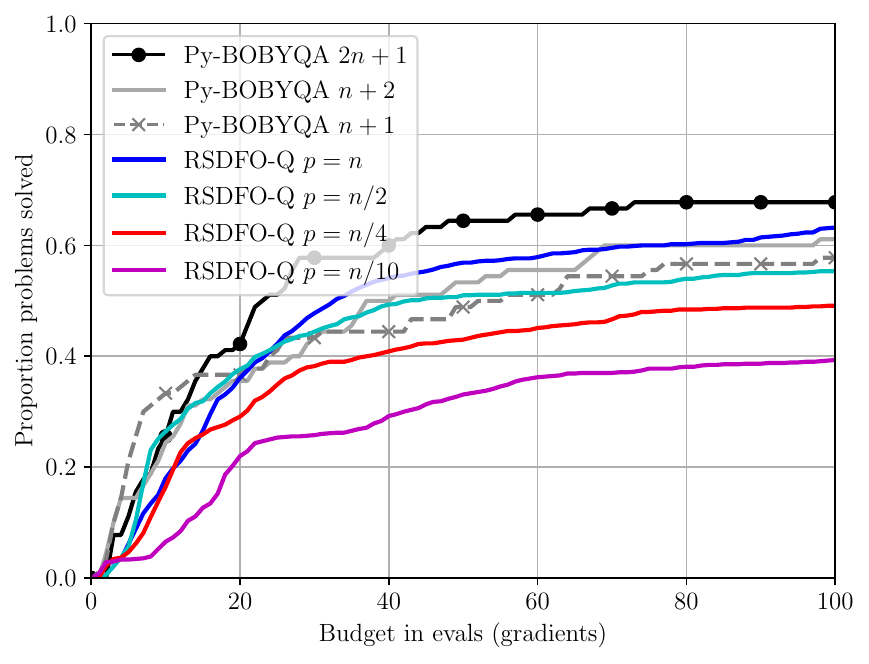}}}\hspace{5pt}
\subfloat[Performance profile, $\tau=10^{-3}$.]{%
\resizebox*{7cm}{!}{\includegraphics{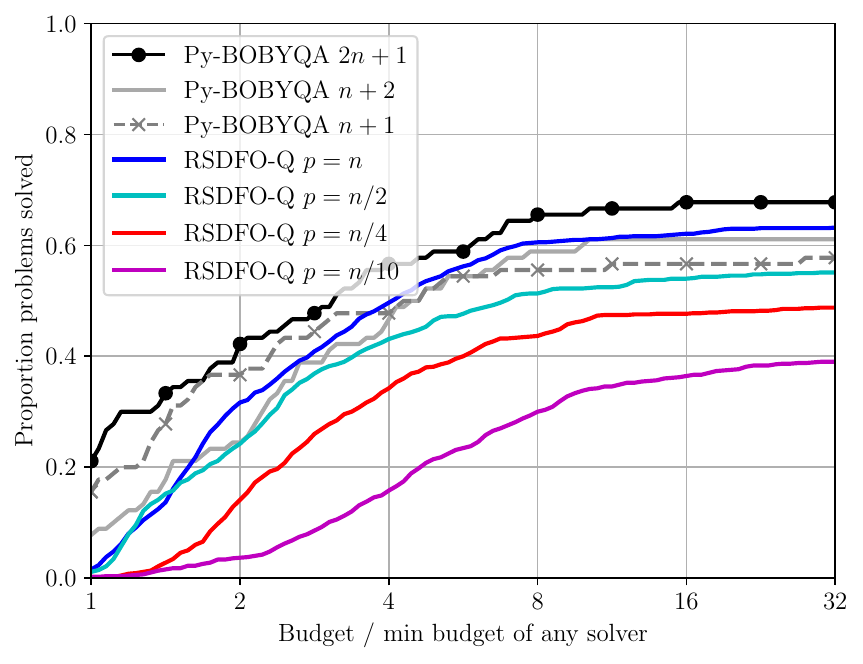}}}
\caption{Comparisons for medium-scale problems, $n\approx 100$.} \label{fig_cutest100}
\end{figure}

\subsection{Large-Scale results}
In Figure~\ref{fig_cutest1000} we show the same results (based on total objective evaluations) as Figure~\ref{fig_cutest100}, but for the large-scale test set `CR-large' ($n\approx 1000$).
Here, we observe that Py-BOBYQA with quadratic models ($q=2n+1$ and $q=n+2$) fail to solve any problems, as their significant per-iteration linear algebra costs\footnote{e.g.~Model construction in each iteration of Py-BOBYQA with $q>n+1$ requires the solution of a $(q+n+1)\times(q+n+1)$ linear system.} mean they hit the runtime limit quickly.
The only Py-BOBYQA variant able to make progress uses linear models with $q=n+1$.
However, this method is outperformed by all variants of the subspace method RSDFO-Q (including with $p=n$).
We find that medium-dimensional subspaces $p\in\{n/10,n/4\}$ (i.e.~$p\approx 100$ and $p\approx 250$) are the best-performing variants.

For comparison, in Figure~\ref{fig_cutest1000_runtime}, we show data and performance profiles based on runtime to find a point with sufficient decrease.
Among the variants of RSDFO-Q, we find that the low-dimensional variants $p\in\{n/10, n/100\}$ are the best-performing, at least initially.
However, the increased robustness of the $p=n/4$ variant (accompanied by an increase in the per-iteration linear algebra cost) means it ultimately outperforms the $p=n/100$ variant.

\begin{figure}[tbh]
\centering
\subfloat[Data profile, $\tau=10^{-1}$.]{%
\resizebox*{7cm}{!}{\includegraphics{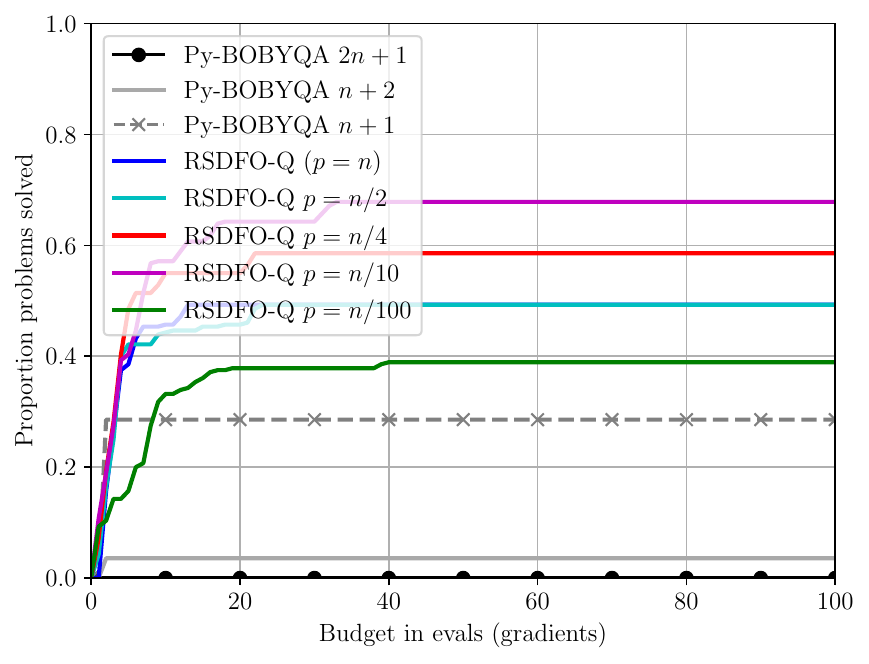}}}\hspace{5pt}
\subfloat[Performance profile, $\tau=10^{-1}$.]{%
\resizebox*{7cm}{!}{\includegraphics{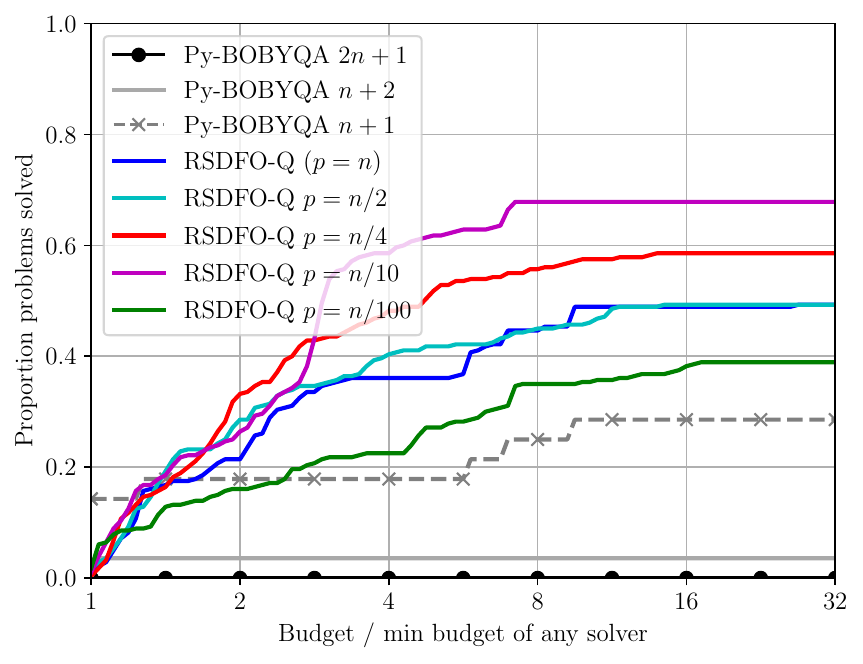}}}
\\
\subfloat[Data profile, $\tau=10^{-3}$.]{%
\resizebox*{7cm}{!}{\includegraphics{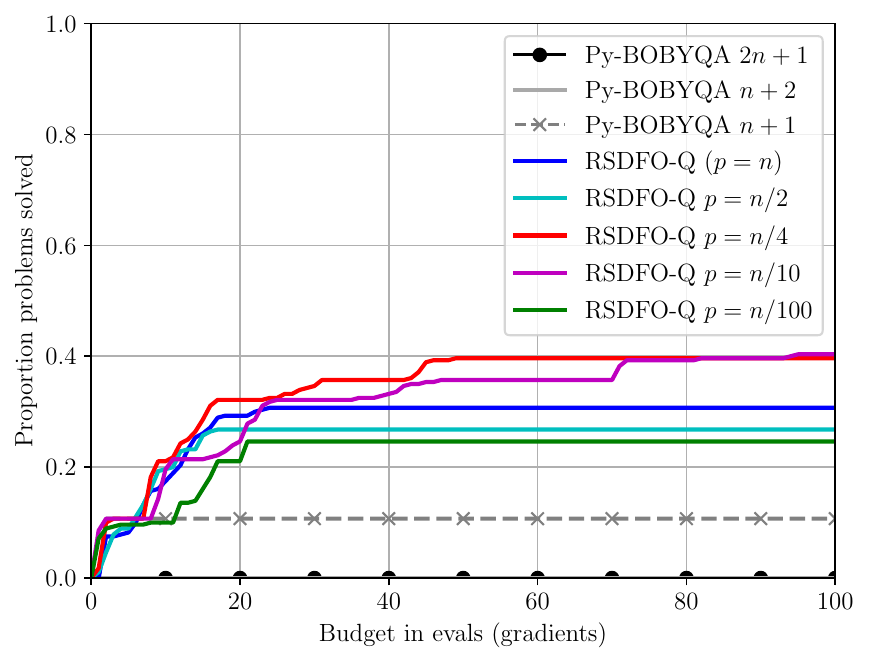}}}\hspace{5pt}
\subfloat[Performance profile, $\tau=10^{-3}$.]{%
\resizebox*{7cm}{!}{\includegraphics{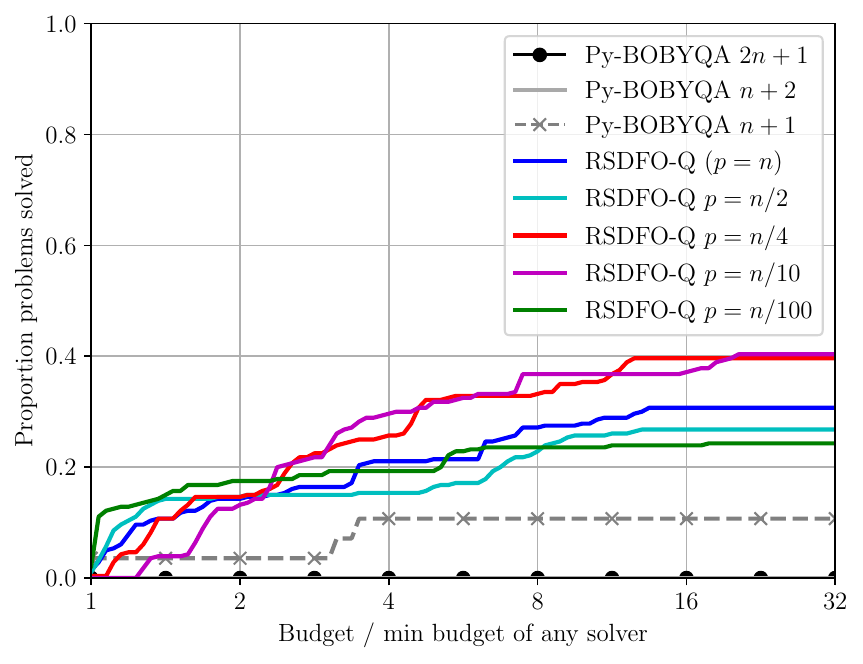}}}
\caption{Comparisons for large-scale problems, $n\approx 1000$. Runtime capped at 12hrs per problem instance for all solvers.} \label{fig_cutest1000}
\end{figure}

\begin{figure}[tbh]
\centering
\subfloat[Data profile, $\tau=10^{-1}$.]{%
\resizebox*{7cm}{!}{\includegraphics{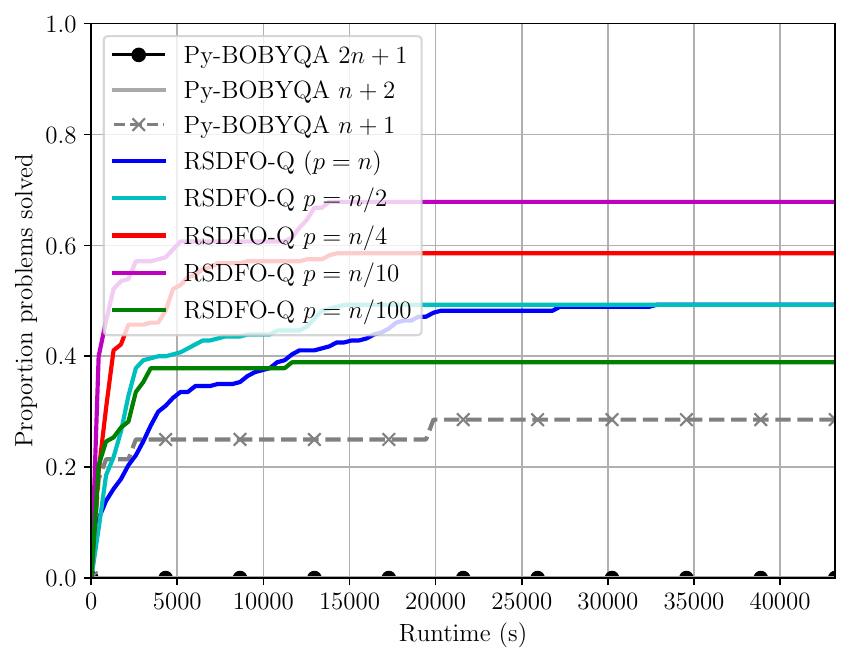}}}\hspace{5pt}
\subfloat[Performance profile, $\tau=10^{-1}$.]{%
\resizebox*{7cm}{!}{\includegraphics{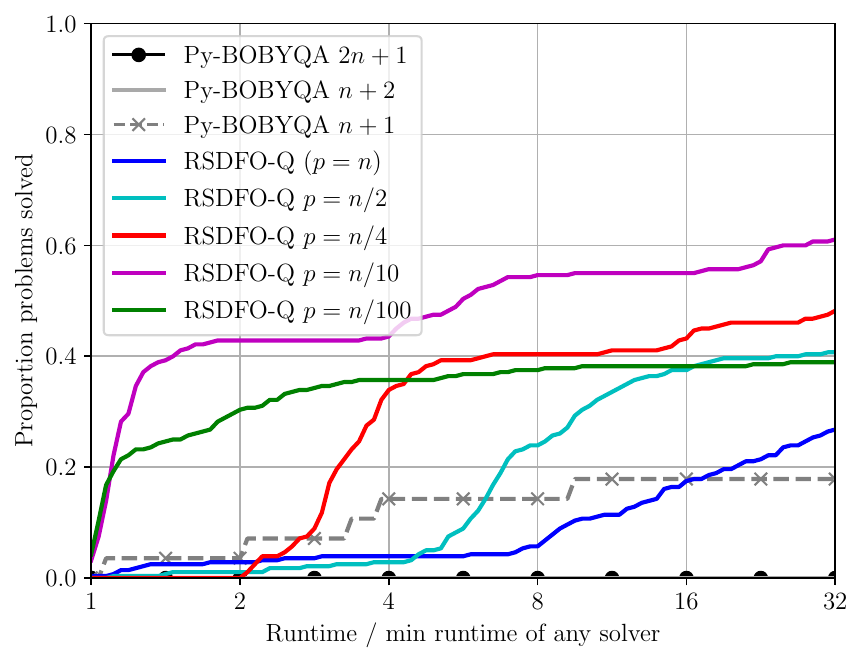}}}
\\
\subfloat[Data profile, $\tau=10^{-3}$.]{%
\resizebox*{7cm}{!}{\includegraphics{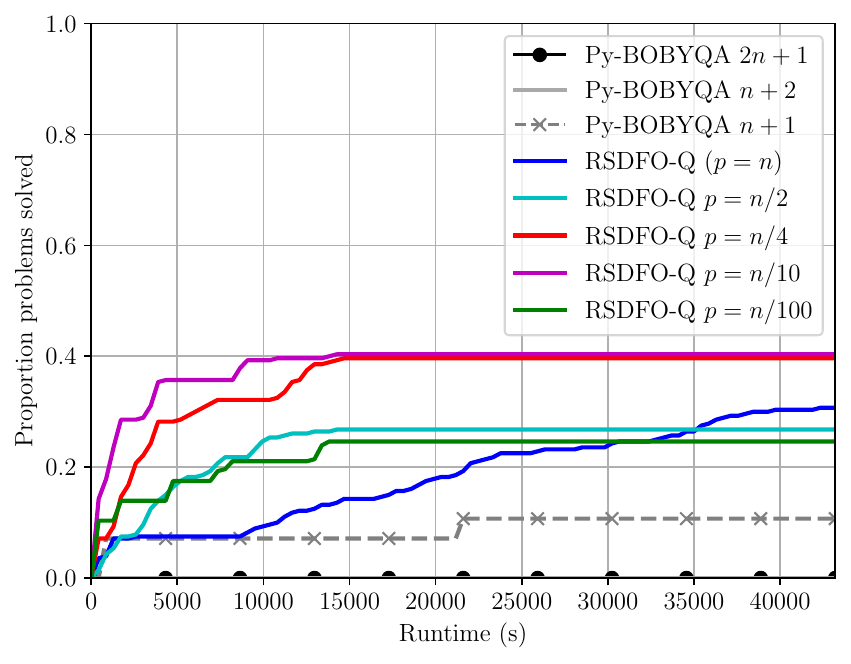}}}\hspace{5pt}
\subfloat[Performance profile, $\tau=10^{-3}$.]{%
\resizebox*{7cm}{!}{\includegraphics{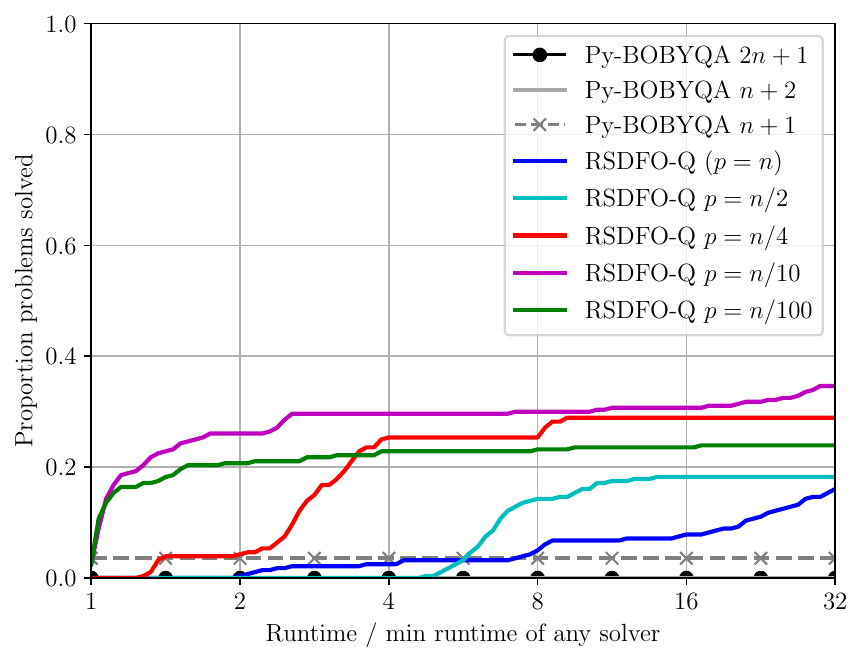}}}
\caption{Comparisons for large-scale problems, $n\approx 1000$, based on runtime required to solve the problem (rather than objective evaluations). Runtime capped at 12hrs per problem instance for all solvers.} 
\label{fig_cutest1000_runtime}
\end{figure}

\section{Conclusion}
We have provided the first second-order convergence result for a random subspace model-based DFO method, and shown that, compared to full space model-based DFO, these approaches can achieve a significant improvement in dimension dependence of the complexity bound, from $O(n^{11} \epsilon^{-3})$ evaluations to $\tilde{O}(n^{4.5} \epsilon^{-3})$ evaluations to find an $\epsilon$-approximate second-order critical point for an $n$-dimensional problem.
This theory is relevant particularly when seeking low accuracy solutions or if the objective has low effective rank (i.e.~low rank Hessians).

We then introduce a practical algorithm (RSDFO-Q) for general unconstrained model-based DFO in subspaces based on iterative construction of quadratic models in subspaces.
This extends the practical method from \cite{Cartis2023}, which was limited to linear models (and specialized to nonlinear least-squares problems).
Numerical results compared to the state-of-the-art method Py-BOBYQA show that RSDFO-Q can solve problems of significantly higher dimension ($n\approx 1000$), exploiting a significantly reduced per-iteration linear algebra cost, while RSDFO-Q with full-dimensional subspaces achieves comparable performance to Py-BOBYQA on medium-scale problems ($n\approx 100$).

Interesting extensions of this approach would be to consider the case of stochastic function evaluations, similar to \cite{Dzahini2024b} for the first-order theory, and a practical comparison of RSDFO-Q for large language model fine-tuning problems, which typically have low effective rank \cite{Malladi2023}.

\section*{Acknowledgement(s)}
We acknowledge the use of the University of Oxford Advanced Research Computing (ARC) facility\footnote{\url{http://dx.doi.org/10.5281/zenodo.22558}} in carrying out this work.


\section*{Funding}

LR was supported by the Australian Research Council Discovery Early Career Award DE240100006. CC was supported by the Hong Kong Innovation and Technology Commission (InnoHK CIMDA Project).




\bibliographystyle{tfs}
\bibliography{refs}

\begin{thebibliography}{10}
\providecommand{\MR}{\relax\unskip\space MR }
\providecommand{\url}[1]{\normalfont{#1}}
\providecommand{\urlprefix}{Available at }

\bibitem{Alarie2021}
S. Alarie, C. Audet, A.E. Gheribi, M. Kokkolaras, and S. Le~Digabel, \emph{Two
  decades of blackbox optimization applications}, EURO Journal on Computational
  Optimization 9 (2021), p. 100011.

\bibitem{Audet2017}
C. Audet and W. Hare, \emph{Derivative-Free and Blackbox Optimization},
  Springer Series in Operations Research and Financial Engineering, Springer,
  Cham, Switzerland, 2017.

\bibitem{Blanchet2019}
J. Blanchet, C. Cartis, M. Menickelly, and K. Scheinberg, \emph{Convergence
  rate analysis of a stochastic trust-region method via supermartingales},
  INFORMS Journal on Optimization 1 (2019), pp. 92--119.

\bibitem{Cartis2019b}
C. Cartis, J. Fiala, B. Marteau, and L. Roberts, \emph{Improving the
  flexibility and robustness of model-based derivative-free optimization
  solvers}, ACM Transactions on Mathematical Software 45 (2019), pp. 1--41.

\bibitem{Cartis2022Zhen}
C. Cartis, J. Fowkes, and Z. Shao, \emph{Randomised subspace methods for
  non-convex optimization, with applications to nonlinear least-squares}, arXiv
  preprint arXiv:2211.09873  (2022).

\bibitem{Cartis-Massart2024}
C. Cartis, X. Liang, E. Massart, and A. Otemissov, \emph{Learning the subspace
  of variation for global optimization of functions with low effective
  dimension}, arXiv preprint arXiv:2401.17825  (2024).

\bibitem{Cartis2022}
C. Cartis and A. Otemissov, \emph{A dimensionality reduction technique for
  unconstrained global optimization of functions with low effective
  dimensionality}, Information and Inference: A Journal of the IMA 11 (2022),
  pp. 167--201.

\bibitem{Cartis2019}
C. Cartis and L. Roberts, \emph{A derivative-free {G}auss--{N}ewton method},
  Mathematical Programming Computation 11 (2019), pp. 631--674.

\bibitem{Cartis2023}
C. Cartis and L. Roberts, \emph{Scalable subspace methods for derivative-free
  nonlinear least-squares optimization}, Mathematical Programming 199 (2023),
  pp. 461--524.

\bibitem{Cartis2022Escaping}
C. Cartis, L. Roberts, and O. Sheridan-Methven, \emph{Escaping local minima
  with local derivative-free methods: A numerical investigation}, Optimization
  71 (2022), pp. 2343--2373.

\bibitem{Cartis2024}
C. Cartis, Z. Shao, and E. Tansley, \emph{Random subspace cubic regularization
  methods}, December  (2024).

\bibitem{Chen2024}
Y. Chen, W. Hare, and A. Wiebe, \emph{Q-fully quadratic modeling and its
  application in a random subspace derivative-free method}, Computational
  Optimization and Applications  (2024).

\bibitem{Conn2009}
A.R. Conn, K. Scheinberg, and L.N. Vicente, \emph{Introduction to
  Derivative-Free Optimization}, MPS-SIAM Series on Optimization, Society for
  Industrial and Applied Mathematics, Philadelphia, 2009.

\bibitem{Dolan2002}
E.D. Dolan and J.J. Mor{\'e}, \emph{Benchmarking optimization software with
  performance profiles}, Mathematical Programming 91 (2002), pp. 201--213.

\bibitem{Dzahini2024}
K.J. Dzahini and S.M. Wild, \emph{Direct search for stochastic optimization in
  random subspaces with zeroth-, first-, and second-order convergence and
  expected complexity}, arXiv preprint arXiv:2403.13320  (2024).

\bibitem{Dzahini2024b}
K.J. Dzahini and S.M. Wild, \emph{Stochastic trust-region algorithm in random
  subspaces with convergence and expected complexity analyses}, SIAM Journal on
  Optimization 34 (2024), pp. 2671--2699.

\bibitem{Ehrhardt2021}
M.J. Ehrhardt and L. Roberts, \emph{Inexact derivative-free optimization for
  bilevel learning}, Journal of Mathematical Imaging and Vision 63 (2021), pp.
  580--600.

\bibitem{Fowkes2022}
J. Fowkes, L. Roberts, and {\'A}. B{\H u}rmen, \emph{{PyCUTEst}: An open source
  {P}ython package of optimization test problems}, Journal of Open Source
  Software 7 (2022), p. 4377.

\bibitem{Garmanjani2015}
R. Garmanjani, \emph{Trust-region methods without using derivatives: Worst case
  complexity and the nonsmooth case}, Ph.D. diss., Universidade de Coimbra,
  2015.

\bibitem{Gould2015}
N.I.M. Gould, D. Orban, and P.L. Toint, \emph{{CUTEst}: A constrained and
  unconstrained testing environment with safe threads for mathematical
  optimization}, Computational Optimization and Applications 60 (2015), pp.
  545--557.

\bibitem{Gratton2015}
S. Gratton, C.W. Royer, L.N. Vicente, and Z. Zhang, \emph{Direct search based
  on probabilistic descent}, SIAM Journal on Optimization 25 (2015), pp.
  1515--1541.

\bibitem{Kimiaei2023}
M. Kimiaei, A. Neumaier, and P. Faramarzi, \emph{New subspace method for
  unconstrained derivative-free optimization}, ACM Transactions on Mathematical
  Software 49 (2023), pp. 1--28.

\bibitem{Kozak2021}
D. Kozak, S. Becker, A. Doostan, and L. Tenorio, \emph{A stochastic subspace
  approach to gradient-free optimization in high dimensions}, Computational
  Optimization and Applications 79 (2021), pp. 339--368.

\bibitem{Larson2019}
J. Larson, M. Menickelly, and S.M. Wild, \emph{Derivative-free optimization
  methods}, Acta Numerica 28 (2019), pp. 287--404.

\bibitem{Malladi2023}
S. Malladi, T. Gao, E. Nichani, A. Damian, J.D. Lee, D. Chen, and S. Arora,
  \emph{Fine-Tuning Language Models with Just Forward Passes}, in \emph{37th
  Conference on Neural Information Processing Systems ({NeurIPS} 2023)}. 2023.

\bibitem{Menickelly2023}
M. Menickelly, \emph{Avoiding geometry improvement in derivative-free
  model-based methods via randomization}, arXiv preprint arXiv:2305.17336
  (2023).

\bibitem{More2009}
J.J. Mor{\'e} and S.M. Wild, \emph{Benchmarking derivative-free optimization
  algorithms}, SIAM Journal on Optimization 20 (2009), pp. 172--191.

\bibitem{Powell2004}
M.J.D. Powell, \emph{Least {F}robenius norm updating of quadratic models that
  satisfy interpolation conditions}, Mathematical Programming 100 (2004), pp.
  183--215.

\bibitem{Powell2009}
M.J.D. Powell, \emph{The {BOBYQA} algorithm for bound constrained optimization
  without derivatives}, Tech. {R}ep. DAMTP 2009/NA06, University of Cambridge,
  2009.

\bibitem{Roberts2023}
L. Roberts and C.W. Royer, \emph{Direct search based on probabilistic descent
  in reduced spaces}, SIAM Journal on Optimization 33 (2023).

\bibitem{Shao2021}
Z. Shao, \emph{On random embeddings and their application to optimisation},
  Ph.D. diss., University of Oxford,  2021.

\bibitem{Ughi2022}
G. Ughi, V. Abrol, and J. Tanner, \emph{An empirical study of
  derivative-free-optimization algorithms for targeted black-box attacks in
  deep neural networks}, Optimization and Engineering 23 (2022), pp.
  1319--1346.

\bibitem{VanDeBerg2024}
D. van~de  Berg, N. Shan, and A. del  Rio-Chanona, \emph{High-dimensional
  derivative-free optimization via trust region surrogates in linear
  subspaces}, in \emph{Proceedings of the 34th European Symposium on Computer
  Aided Process Engineering / 15th International Symposium on Process Systems
  Engineering (ESCAPE34/PSE24)}. 2024.

\bibitem{Woodruff2014}
D.P. Woodruff, \emph{Sketching as a tool for numerical linear algebra},
  Foundations and Trends{\textregistered} in Theoretical Computer Science 10
  (2014), pp. 1--157.

\bibitem{Xie2024}
P. Xie and Y.x. Yuan, \emph{A new two-dimensional model-based subspace method
  for large-scale unconstrained derivative-free optimization: {2D-MoSub}},
  arXiv preprint arXiv:2309.14855  (2024).

\bibitem{Zhang2012}
Z. Zhang, \emph{On derivative-free optimization methods}, Ph.D. diss., Chinese
  Academy of Sciences,  2012.

\end{thebibliography}


\appendix






\end{document}